\documentclass[a4paper,reqno,11pt]{amsart}

\usepackage{amsmath, amsfonts, amssymb, amsthm, amscd}
\usepackage{graphicx}
\usepackage{psfrag}
\usepackage{perpage}
\usepackage{url}
\usepackage{color}

\usepackage[utf8]{inputenc}
\usepackage[T1]{fontenc}
\usepackage{microtype}

\usepackage[a4paper,scale={0.72,0.74},marginratio={1:1},footskip=7mm,headsep=10mm]{geometry}

\usepackage{hyperref}

\numberwithin{equation}{section}

\newtheorem{theorem}{Theorem}[section]
\newtheorem{lemma}[theorem]{Lemma}

\newtheorem{remark}[theorem]{Remark}

\newcommand{\bbE}{{\ensuremath{\mathbb E}} }

\newcommand{\bbN}{{\ensuremath{\mathbb N}} }

\newcommand{\bbP}{{\ensuremath{\mathbb P}} }

\newcommand{\bbR}{{\ensuremath{\mathbb R}} }

\newcommand{\cN}{{\ensuremath{\mathcal N}} }

\newcommand{\ga}{\alpha}
\newcommand{\gb}{\beta}

\newcommand{\gd}{\delta}

\newcommand{\gz}{\zeta}

\newcommand{\gl}{\lambda}
\newcommand{\gL}{\Lambda}

\newcommand{\go}{\omega}

\renewcommand{\tilde}{\widetilde}          
\DeclareMathSymbol{\leqslant}{\mathalpha}{AMSa}{"36} 
\DeclareMathSymbol{\geqslant}{\mathalpha}{AMSa}{"3E} 
\DeclareMathSymbol{\eset}{\mathalpha}{AMSb}{"3F}     
\newcommand{\dd}{\text{\rm d}}             

\newcommand{\sumtwo}[2]{\sum_{\substack{#1 \\ #2}}} 

\newcommand{\R}{\mathbb{R}}

\newcommand{\Z}{\mathbb{Z}}
\newcommand{\N}{\mathbb{N}}

\def\bs{\boldsymbol}

\newcommand{\PEfont}{\mathrm}

\def\p{\ensuremath{\PEfont P}}
\def\e{\ensuremath{\PEfont E}}
\newcommand{\E}{\e}
\renewcommand{\P}{\p}

\newcommand\bE{\ensuremath{\bs{\mathrm{E}}}}

\DeclareMathOperator{\bbvar}{\ensuremath{\mathbb{V}ar}}

\newcommand{\ind}{{\sf 1}}

\renewcommand{\epsilon}{\varepsilon}
\renewcommand{\theta}{\vartheta}
\renewcommand{\rho}{\varrho}
\renewcommand{\phi}{\varphi}

\newenvironment{myenumerate}{%
\renewcommand{\theenumi}{\arabic{enumi}}%
\renewcommand{\labelenumi}{{\rm(\theenumi)}}%
\begin{list}{\labelenumi}
	{%
	\setlength{\itemsep}{0.4em}%
	\setlength{\topsep}{0.5em}%
	\setlength\leftmargin{2.45em}%
	\setlength\labelwidth{2.05em}%
	\setlength{\labelsep}{0.4em}%
	\usecounter{enumi}%
	}%
	}%
{\end{list}
}

{\end{myenumerate}}

\newenvironment{myitemize}{%
\begin{list}{$\bullet$}%
 	{%
	\setlength{\itemsep}{0.4em}%
	\setlength{\topsep}{0.5em}%
	\setlength\leftmargin{2.45em}%
	\setlength\labelwidth{2.05em}%
	\setlength{\labelsep}{0.4em}%
	}%
	}%
{\end{list}}

\renewenvironment{itemize}{
\begin{myitemize}}%
{\end{myitemize}}

\MakePerPage[2]{footnote} 

\begin{document}

\begin{abstract}
We study random pinning and copolymer models,
when the return distribution of the
underlying renewal process has a polynomial tail with finite mean.
We compute the asymptotic behavior of the critical curves of the models
in the weak coupling regime, showing that it is universal.
This proves a conjecture of Bolthausen, den Hollander and Opoku
for copolymer models \cite{cf:BdHO}, which we also extend to pinning models.
\end{abstract}

\title[Critical Curves at Weak Coupling]{The Critical Curves of
the Random Pinning and Copolymer Models at Weak Coupling}

\author[Q.Berger]{Quentin Berger}
\address{Department of Mathematics, KAP 108\\
University of Southern California\\
Los Angeles, CA  90089-2532 USA}
\email{qberger@usc.edu}

\author[F.Caravenna]{Francesco Caravenna}
\address{Dipartimento di Matematica e Applicazioni\\
 Universit\`a degli Studi di Milano-Bicocca\\
 via Cozzi 53, 20125 Milano, Italy}
\email{francesco.caravenna@unimib.it}

\author[J.Poisat]{Julien Poisat}
\address{
Mathematical Institute,
Leiden University\\
 P.O. Box 9512, NL-2300 RA Leiden\\
  The Netherlands
}
\email{poisatj@math.leidenuniv.nl}
\author[R.Sun]{Rongfeng Sun}
\address{
Department of Mathematics\\
National University of Singapore\\
10 Lower Kent Ridge Road, 119076 Singapore
}
\email{matsr@nus.edu.sg}
\author[N.Zygouras]{Nikos Zygouras}
\address{Department of Statistics\\
University of Warwick\\
Coventry CV4 7AL, UK}
\email{N.Zygouras@warwick.ac.uk}
\date{\today}

\keywords{Pinning Model, Copolymer Model, Critical Curve, Fractional Moments, Coarse Graining}
\subjclass{Primary: 82B44; Secondary: 82D60, 60K35}

\maketitle

\def\rc{\mathrm{c}}
\def\dd{\mathrm{d}}

\def\hZ{{\hat{Z}}}
\def\hbeta{{\hat{\beta}}}
\def\hh{{\hat{h}}}
\def\homega{{\hat{\omega}}}
\def\hC{{\hat{C}}}
\def\hlambda{{\hat{\lambda}}}

\def\f{\textsc{f}}
\def\hf{\hat{\textsc{f}}}

\def\cop{\mathrm{cop}}


\section{Introduction}

The presence of disorder can drastically change the statistical mechanical properties
of a physical system and alter the nature of its phase transitions,
leading to new phenomena. Using a random walk to model a polymer chain, the effect of disorder on
random walk models has been of particular interest recently, giving rise to many random
polymer models~\cite{cf:dH, cf:G1}. In this work we focus on two important classes
of random polymer models, the so-called pinning and copolymer models.
\begin{itemize}
\item In the pinning model, the disorder is attached to a defect 
line, which can either attract or repel the random walk path.
As the temperature varies, a localization-delocalization phase transition takes place:
for sufficiently low temperatures, the random walk is absorbed at the defect line,
while at high temperature it wanders away.
The origins of such a model can be traced back to studies 
of wetting phenomena \cite{cf:FLNO, cf:DHV} or localization of flux lines in superconducting 
vortex arrays \cite{cf:NV}. We also mention the interesting phenomenon
of DNA denaturation, cf. \cite{cf:PS}.

\item In the copolymer model, disorder is distributed along the random walk path, 
which meanders between two solvents separated by a flat interface. 
Each step of the random walk path can be regarded as 
a monomer, and the disorder attached to the monomer determines whether it prefers one solvent 
or the other. Also for this model, a sharp
localization-delocalization phenomenon is observed:
the typical random walk
paths are either very close to the interface (in order to place most monomers in their preferred 
solvents) or very far from it, according to the temperature. 
The origins of the copolymer model in this context can be traced back to \cite{cf:GHLO}. 
\end{itemize}

The purpose of this paper is to investigate the 
phase diagram of both models in the weak coupling regime,
in the case when the excursions of the random walk away from the defect line
(or interface) have a power-law tail with finite mean.

\subsection{Review of the models}
We first recall the definition of the random pinning and copolymer models. For a general 
overview, we refer to~\cite{cf:dH, cf:G1, cf:G2, cf:CGT}. 

The polymer chain is modeled by a Markov chain
$S = \{S_n\}_{n\geq 0}$ on $\Z$ with $S_0=0$, that will be called \emph{the walk}. Probability 
and expectation for $S$ will be denoted respectively by $\p$ and $\e$. 
We denote by $\tau:=\{\tau_n\}_{n\geq 0}$, with
$0 = \tau_0 < \tau_1 < \tau_2 < \ldots$, the sequence of random times in which the walk 
visits $0$, so that $\tau$ is a \emph{renewal process} with $\tau_0=0$.
We assume that $\tau$ is \emph{non-terminating}, that is $\p(\tau_1 < \infty) = 1$, and that
\begin{equation} \label{eq:ass}
	K(n):=\p(\tau_1=n)=\frac{\phi(n)}{n^{1+\alpha}}\,, \quad \forall n\in \N = \{1, 2, \ldots\} \,,
\end{equation}
where $\alpha\in [0,+\infty)$ and $\phi: (0,\infty) \to (0,\infty)$ is a slowly varying function. 
In this paper \emph{we focus on the case $\alpha > 1$,
for which the mean return time is finite}:
\begin{equation}\label{eq:mean}
	\mu \,:=\, \e[\tau_1] \,\in\, (1, \infty) \,.
\end{equation}

\begin{remark}\rm
Many interesting examples have periodicity issues, that is, there exists $\text{\sc t} \in \N$
such that $K(n) = 0$ if $n \not\in \text{\sc t}\N$.
For instance, the simple symmetric random walk on $\Z$
satisfies \eqref{eq:ass} for $n\in 2\N$,
with $\alpha=1/2$ and $\phi(\cdot)$ converging asymptotically to a constant.  
However, we shall assume for simplicity that $K(n) > 0$ for every $n\in\N$.
Everything can be easily extended to the periodic case, at the expense of some
cumbersome notation.
\end{remark}

\begin{remark}\rm
As it will be clear, in our framework the fundamental object
is the renewal process $\tau$, and \emph{there is no need to refer to the Markov chain $S$}.
However, let us mention that, for any $\alpha>0$ and any slowly varying function $\varphi$, a 
nearest-neighbor Markov chain $S$ on $\Z$ with Bessel-like drift can be constructed,
which satisfies assumption \eqref{eq:ass} asymptotically, that is
$K(n)\sim \varphi(n)/n^{1+\alpha}$ as $n\to\infty$, cf.~\cite{cf:A}. 
\end{remark}

The disorder is modeled by a sequence $\go:=\{\go_n\}_{n\geq 1}$ of i.i.d.\
real random variables. Probability and expectation for $\go$ will be denoted respectively 
by $\bbP$ and $\bbE$. We assume that
\begin{equation} \label{eq:assdis}
	M(t):=\bbE[e^{t\go_1}] <\infty \quad \forall |t| < t_0,\text{ with } t_0 > 0\,, \qquad
	\bbE[\go_1]=0\,, \qquad \bbvar(\go_1)=1 \,.
\end{equation}
We will often be making use of 
the log-moment generating function, that is
\begin{equation} \label{eq:asLambda}
	\Lambda(t):=\log M(t) = \frac{1}{2} t^2 + o(t^2) \quad \text{ as } t \to 0 \,.
\end{equation}
Given a realization of the disorder $\go$, the random \emph{pinning model} is defined by a Gibbs 
transform of the law $\p$ of the renewal process $\tau$ (or, if one wishes,
of the walk $S$):
 \begin{equation}\label{eq:defpin}
 \dd\p_{N,\beta,h}^{{\rm pin},\go}:=\frac{1}{Z_{N,\gb,h}^{{\rm pin},\go}} \,e^{\sum_{n=1}^N
	(\beta \omega_n - h_a^{{\rm pin}}(\beta) + h) \ind_{\{n\in\tau\}}} \dd\p,
 \end{equation}
where $\{n\in\tau\}$ is a shorthand for $\bigcup_{k\in\N} \{\tau_k = n\}$,
the event that the renewal process $\tau$ visits $n$
(which corresponds to $\{S_n=0\}$, referring to the walk $S$).
The parameter $\beta\geq 0$ is the coupling constant (or 
inverse temperature), $h\in\R$ adds a bias to the disorder, and
\begin{equation}\label{eq:hannpin}
	h_a^{{\rm pin}}(\beta) := \Lambda(\beta) = \frac{\beta^2}{2} + o(\beta^2)
	\qquad \text{(as $\beta \downarrow 0$)}\,
\end{equation}
is the {\it annealed} critical point, the significance of which will be discussed later. 
The normalizing constant
\begin{equation}\label{eq:Zpin}
	Z_{N,\gb,h}^{{\rm pin},\go}:=\E\left[ e^{\sum_{n=1}^N (\beta \omega_n 
	-h_a^{{\rm pin}}(\beta) + h) \ind_{\{n\in\tau\}}} \right]
\end{equation}
is called the \emph{partition function}. 
We will also consider the \emph{constrained} partition function
\begin{equation}\label{eq:Zpinc}
Z_{N,\beta,h}^{{\rm pin},\rc,\omega} \,=\,	\E\Big[ e^{\sum_{n=1}^N(\beta \omega_n 
-h_a^{{\rm pin}}(\beta) + h) \ind_{\{n\in\tau\}}} \ind_{\{N \in \tau\}}\Big] \,.
\end{equation}

In order to define the copolymer model, one traditionally works with nearest-neighbor
walks $S$ that make symmetric excursions in the positive or negative half-plane.
The signs of the excursions are then i.i.d.\ symmetric $\{\pm 1\}$-valued 
random variables, independent of $\tau$. In our framework,
it is actually simpler to proceed as in \cite{cf:CGT}.
Without making any reference to the walk $S$, under the law $\p$ we introduce a sequence of
i.i.d. symmetric $\pm 1$ random variables
$\hat\varepsilon:=\{\hat\varepsilon_n\}_{n\geq 1}$, independent of the renewal process $\tau$, 
which model the sign of the excursion of the walk
during the renewal interval $(\tau_{n-1}, \tau_n)$ (even when this interval
has length $1$).
We also introduce the variables 
$\varepsilon_n=\sum_{k\geq 1} \hat\varepsilon_k 1_{\{ n\in (\tau_{k-1},\tau_k]\}}$, 
which represent the sign of the $n^{th}$ step of the walk.
Given the disorder $\omega$, the \emph{copolymer model} is
then defined via a similar Gibbs transform of the law $\p$:
\begin{equation}\label{eq:cop}
 	\dd\p_{N,\gl,h}^{{\rm cop},\go}:=\frac{1}{Z_{N,\gl,h}^{{\rm cop},\go}} \,
	 e^{-2\lambda \sum_{n=1}^N
	(\omega_n+h^{{\rm cop}}_a(\lambda)- h) \ind_{\{\varepsilon_n = -1\}}} \dd\p,
\end{equation}
where $\lambda\geq 0$ is the coupling constant,
$h\in\R$ and $Z_{N,\gl,h}^{{\rm cop},\go}$ have the same interpretation as in the pinning model, and
\begin{equation}\label{eq:hanncop}
	h_a^{{\rm cop}}(\lambda) := \frac{1}{2\gl} \, \Lambda(-2\gl) =
	\lambda + o(\lambda) \quad \text{(as $\lambda \downarrow 0$)}\,
\end{equation}
is the corresponding {\it annealed} critical point. 
The constrained partition function for the copolymer model is given by
 \begin{equation}\label{eq:Zcopc}
\begin{split}
	Z_{N,\gl,h}^{{\rm cop},\rc,\omega} & \,=\,
	\E\Big[ e^{-2\gl\sum_{n=1}^N(\omega_n +h_a^{{\rm cop}}(\lambda) 
	- h) \ind_{\{\varepsilon_n = -1\}} }
	\ind_{\{N\in \tau\}}\Big] \\
	& \,=\,
	\E\Bigg[ \prod_{j=1}^{|\tau \cap (0,N]|}
	\bigg( \frac{1 + e^{-2\gl\sum_{n=\tau_{j-1}+1}^{\tau_j} 
	(\omega_n +h_a^{{\rm cop}}(\lambda) - h) }}{2} \bigg)
	\ind_{\{N\in \tau\}}\Bigg] \,,
\end{split}
\end{equation}
where $|\tau \cap (0,N]| = \max\{k \ge 0: \, \tau_k \le N\}$ is the number of renewal
points that appear before $N$, and we have integrated the excursions signs.

\begin{remark}\rm\label{rem:para}
The different parametrization of the copolymer model, as compared to the pinning one,
is to conform with most of the existing literature.
To recover the pinning parametrization, it suffices to replace $\omega$ by $-\omega$,
$2\gl$ by $\beta$ and $2\lambda h$ by $h$.
\end{remark}

Many statistical properties of the models can be captured through the (quenched) \emph{free energies}, 
which are defined by
\begin{align} \label{eq:free}
\begin{split}
	\f^{{\rm pin}}(\gb,h) & :=\lim_{N\to\infty}\frac{1}{N}\,\log Z_{N,\gb,h}^{{\rm pin},\go} 
	=\lim_{N\to\infty}\frac{1}{N}\, \bbE \log Z_{N,\gb,h}^{{\rm pin},\go} \,, \\
	\f^{{\rm cop}}(\gl,h) & :=\lim_{N\to\infty}\frac{1}{N}\,\log Z_{N,\gb,h}^{{\rm cop},\go} 
	=\lim_{N\to\infty}\frac{1}{N}\, \bbE \log Z_{N,\gl,h}^{{\rm cop},\go} \,,
\end{split}
\end{align}
where the limits exist $\bbP$-a.s.\ and remain unchanged if we replace the partition functions by 
their constrained counterparts (see \cite[Ch. 4]{cf:G1}).
From the definition of the partition functions, by restricting the expectation $\e$
to the event $\{\tau_1 > N\}$ (that is $\{S_n> 0 \ \text{for}\  n=1,2,...,N\}$
in the walk interpretation) and observing that $\log \p(\tau_1 > N) = O(\log N)$, by \eqref{eq:ass},
it follows that the free energies are nonnegative.
A (quenched) localization-delocalization transition can be determined from the critical curves
\begin{equation}
h_c^{{\rm pin}}(\gb):=\sup\{h\colon \ \f^{{\rm pin}}(\gb,h)=0\} \quad \mbox{and} \quad
h_c^{{\rm cop}}(\gl):=\sup\{h \colon \ \f^{{\rm cop}}(\gl,h)=0\}.
\end{equation}
The meaning of these critical curves is that, when $h$ exceeds the critical value,
the free energy is strictly positive and the polymer 
puts a positive fraction of its monomers at $\{0\}$, in the pinning model
(resp. in $\{0,-1,-2,...\}$, in the copolymer model);
on the other hand, these fractions equal $0$ when $h$ is below the critical value.
We refer to \cite{cf:G1,cf:dH} for details.

The effect of the disorder is best seen through comparison of the (quenched) models with 
their {\it annealed} counterparts. In particular, the annealed free energies are defined by
\begin{eqnarray*}
	\f^{{\rm pin}}_a(\gb,h):=\lim_{N\to\infty}\frac{1}{N}\,\log \bbE Z_{N,\gb,h}^{{\rm pin},\go} 
	\quad 	\text{and} \quad \f^{{\rm cop}}_a(\gl,h):=\lim_{N\to\infty}\frac{1}{N}\,
	\log \bbE Z_{N,\gb,h}^{{\rm cop},\go} \,.
\end{eqnarray*}
The annealed models are exactly solvable and their phase diagrams can be completely determined 
(see e.g.~\cite{cf:G1}). In particular, the critical value of $h$, above
which the annealed free energy is strictly positive, is $h = 0$, both for the pinning
and copolymer model. This is simply 
because we have subtracted the ``true'' annealed critical points $h_a^{{\rm pin}}(\beta)$ and
$h_a^{{\rm cop}}(\beta)$ from $h$ in the definition of our models, for later convenience. 
For the pinning model, it turns out that when the quenched critical curve $h_c^{{\rm pin}}(\beta)$ 
is \emph{strictly positive} for any small $\beta$, the order of the phase transition is strictly larger 
(except for possibly the marginal case $\alpha=1/2$) than that of the annealed model, in which 
case disorder is said to be {\em relevant}. On the other hand, when $h_c^{{\rm pin}}(\beta)=0$ for 
small values of $\beta$, the order of the phase transition does not change from that of the 
annealed model, in which case disorder is said to be {\em irrelevant}.  For a more detailed discussion 
on disorder relevance vs irrelevance, see \cite{cf:G2}.


It has been shown in~\cite{cf:BGLT,cf:T1} that for the copolymer model disorder is 
relevant for every $\alpha > 0$,
regardless of the underlying renewal process $\tau$ (satisfying the above
mentioned assumptions). On the other hand, for the pinning model, it is known that disorder 
is irrelevant when $\ga<1/2$ or 
when $\ga=1/2$ and $\sum_{n\geq 1}1/n\phi(n)^2<\infty$ (cf. \cite{cf:A2,cf:AZ2,cf:T2,cf:CdH,cf:L}),
relevant when $\alpha>1/2$, and believed to be also relevant (almost confirmed 
in~\cite{cf:AZ,cf:GLT1,cf:GLT2}) when $\ga=1/2$ and $\sum_{n\geq 1}1/n\phi(n)^2=\infty$. 
See~\cite{cf:G2} for an overview.

\subsection{The main results}

A fundamental problem for random pinning and copolymer models, 
when disorder is relevant, is the asymptotic behavior of the critical curves
in the \emph{weak coupling regime $\beta, \lambda\downarrow 0$}. 
The interest of this question lies in the belief
that such asymptotic behavior should be \emph{universal}, i.e., not depend
too much on the fine details of the model.

\smallskip

For the copolymer model, the behavior of the critical curve $h_c^{{\rm cop}}(\gl)$
for small $\lambda$ has been investigated extensively. 
In the seminal paper~\cite{cf:BdH}, Bolthausen and den Hollander
investigated the special copolymer model
in which $S = \{S_n\}_{n\geq 1}$ is the simple symmetric random walk on $\Z$
and the disorder variables $\omega_n$ are $\{\pm 1\}$-valued and symmetric.
They were able to show that the slope of the critical 
curve $\lim_{\gl\downarrow 0} h_c^{{\rm cop}}(\gl)/\gl$ exists
and coincides with the critical point of a \emph{continuum copolymer model}
(in which the walk $S$ is replaced by a Brownian motion and the
disorder sequence $\omega$ is replaced by white noise).
This result was recently extended by Caravenna and Giacomin~\cite{cf:CG}:
for the general class of copolymer models that we consider
in this paper, in the case $\ga\in(0,1)$,
it was shown that the slope of the critical curve exists and
is a \emph{universal quantity}, namely it is
the critical point of a suitable $\alpha$-continuum copolymer model.
In particular, the slope \emph{depends only on $\alpha$} and not on
finer details of the renewal process $\tau$ and disorder $\omega$.
For consistency with the literature, we define (recall \eqref{eq:hanncop})
\begin{equation} \label{eq:slope}
	m_\alpha \,:=\, \lim_{\gl\downarrow 0} \frac{h_a^{{\rm cop}}(\lambda)
	- h_c^{{\rm cop}}(\lambda)}{\gl} \,=\, 1 - 
	\lim_{\gl\downarrow 0} \frac{h_c^{{\rm cop}}(\lambda)}{\gl} \,.
\end{equation}

The precise value of $m_\alpha$, in particular for $\alpha=1/2$,
has been a matter of a long debate. It was conjectured
by Monthus in \cite{cf:M},
on the ground of non-rigorous renormalisation arguments, that $m_{1/2} = 2/3$,
and a generalization of the same argument 
yields the conjecture
$m_\alpha = 1/(\ga+1)$. The rigorous lower bound $m_\alpha \ge 1/(\ga+1)$,
for every $\alpha \ge 0$, was proved by Bodineau and Giacomin in \cite{cf:BG}.
Very recently, it was shown by Bolthausen, den Hollander
and Opoku \cite{cf:BdHO} that this lower bound is strict
for every $\alpha \in (0,\infty)$, thus ruling out Monthus' conjecture
(see also \cite{cf:BGLT} for earlier, partial results, and
\cite{cf:CGG} for a related numerical study).

In this work, we focus on the case $\alpha > 1$. Let us stress that this case was not considered
in \cite{cf:CG}, because no non-trivial continuum model is expected to exist, due to the finite mean
of the underlying renewal process. This consideration might even cast doubts on the 
(existence and) universality of the limit in \eqref{eq:slope}.
However, it was recently proved in \cite{cf:BdHO} that 
\begin{equation} \label{eq:liminff}
	\liminf_{\gl\downarrow 0} \frac{h_a^{{\rm cop}}(\lambda)-h_c^{{\rm cop}}(\lambda)}{\gl}
	\,\ge\, \frac{2+\ga}{2(1+\ga)} \,, \qquad \forall \alpha \in (1,\infty) \,.
\end{equation}
(The ``critical curve'' in \cite{cf:BdHO} corresponds to 
$h_a^{{\rm cop}}(\lambda)-h_c^{{\rm cop}}(\lambda)$ in our notation;
furthermore, our exponent $\alpha$ is what they call $\alpha - 1$, hence the right
hand side in \eqref{eq:liminff} reads $\frac{1+\alpha}{2\alpha}$ in~\cite{cf:BdHO}.)

The universal lower bound \eqref{eq:liminff}, depending only on $\alpha$, led naturally
to the conjecture \cite{cf:BdHO} that $m_\alpha$ exists also for $\alpha > 1$
and coincides with the right hand side of \eqref{eq:liminff}.
Our first main result proves this conjecture,
establishing in particular the universality of the slope.

\begin{theorem}\label{thmcop}
For any copolymer model defined as above, with $\alpha>1$, the limit
in \eqref{eq:slope} exists and equals $m_\alpha = \frac{2+\ga}{2(1+\ga)}$.
Equivalently,
\begin{equation}\label{eq:copthm}
\lim_{\gl\downarrow 0} \frac{h_c^{{\rm cop}}(\lambda)}{\lambda} =\frac{\ga}{2(1+\ga)}.
\end{equation}
\end{theorem}

\begin{remark}\rm
In a work in progress~\cite{cf:CSZ}, the partition function of 
the copolymer model under weak coupling is shown to converge, for every $\alpha > 1$,
to an explicit ``trivial''
continuum limit, the exponential of a Brownian motion with drift,
\emph{which carries no dependence on $\alpha$}. In particular, 
the continuum limit of the partition function gives no information on the slope of the critical 
curve, which is in stark contrast to the case $\alpha \in (0,1)$.
\end{remark}

For the random pinning model with $\alpha > 1$,
rough upper and lower bounds of the order $\beta^2$ are known for the critical 
curve $h_c^{{\rm pin}}(\beta)$, cf. \cite{cf:AZ,cf:DGLT}.
(The quadratic, rather than linear, behavior is simply
due to the different way the parameters $(\beta, h)$ and $(\lambda, h)$ appear
in the two models, cf. Remark~\ref{rem:para}). 
We sharpen these earlier results by establishing the following analogue of Theorem~\ref{thmcop}.

\begin{theorem}\label{thmpin}
For any random pinning model defined as above, with $\alpha>1$, we have
\begin{equation} \label{eq:CSZ+}
	\lim_{\beta \downarrow 0} \frac{h_c^{{\rm pin}}(\beta)}{\beta^2} \,=\,
	\frac{\alpha}{1+\alpha} \, \frac{1}{2 \mu} \,,
\end{equation}
where $\mu:=\e[\tau_1]$.
\end{theorem}

\noindent
Thus, the asymptotic behavior of the critical curve of the random pinning model
with $\alpha > 1$ is also universal, in the sense that it depends only on the exponent $\alpha$
and on the mean $\mu$ of the underlying renewal process, and not on the finer details of 
the renewal process or the disorder distribution.

\begin{remark}\rm
For the random pinning model with $\alpha>1$, it is also shown in~\cite{cf:CSZ} that the 
partition function under weak coupling converges, in the continuum limit, to the exponential of 
a Brownian motion with drift, which depends on $\mu$ \emph{but not on $\alpha>1$}.
As a consequence, the continuum limit gives no information on the asymptotic
behavior \eqref{eq:CSZ+}.
\end{remark}

The fact that we can prove the same type of result for the random pinning and copolymer models 
is not unexpected for $\alpha > 1$. In fact,
when the underlying renewal process has finite mean, 
there is typically a positive fraction of monomers
interacting with the disorder for both models,
since $\frac{1}{N} \sum_{n=1}^N1_{\{n\in\tau\}}$ and 
$\frac{1}{N} \sum_{n=1}^N 1_{\{\epsilon_n = -1\}}$
(recall \eqref{eq:defpin} and \eqref{eq:cop})
tend to a positive constant as $N\to\infty$ ($1/\mu$ and $1/2$, respectively).
Also, the annealed free energies of both models are proportional 
to the square of the coupling constants $\beta$, respectively $\gl$, if the bias $h$ is also 
scaled properly (proportional to $\gb^2$ for the pinning model and to $\gl$ for the 
copolymer model). This fact is then reflected in the (annealed) correlation length of the 
system, which plays a central role in the coarse graining analysis we will carry out later.

When $\alpha<1$, the above analogies between the random pinning and copolymer models break down. 
In \cite{cf:AZ} it was shown that for $\ga\in(1/2,1)$, $h_c^{{\rm pin}}(\gb)$ satisfies the bounds
\begin{eqnarray*}
c^{-1} \gb^{\frac{2\ga}{2\ga-1}}\psi(\gb^{-2}) \leq h_c^{{\rm pin}}(\gb)\leq 
c \gb^{\frac{2\ga}{2\ga-1}}\psi(\gb^{-2})
\end{eqnarray*}
for some $c>0$ and a slowly varying function $\psi(\cdot)$, related in an explicit way 
to $\phi(\cdot)$. It is then reasonable to conjecture that, in analogy
with the universality of the slope of the critical curve for the copolymer model, the limit
\begin{equation} \label{eq:univ<1}
	\lim_{\gb\downarrow 0} \frac{h_c^{{\rm pin}}(\beta)}{\gb^{\frac{2\ga}{2\ga-1}}\psi(\gb^{-2})}
\end{equation}
exists and is also universal. The method we use in 
this paper falls short in answering this question. However, the random pinning model under 
weak coupling does admit a continuum limit when $\alpha \in (1/2, 1)$, which is currently 
under construction in~\cite{cf:CSZ}. This gives hope to prove the existence and
universality of the limit in \eqref{eq:univ<1}.

\subsection{Organization and main ideas}\label{S:outline}

We present the proof of Theorem~\ref{thmpin}, concerning the random pinning model, in 
Sections~\ref{sec2} (lower bound) and~\ref{sec3} (upper bound). The proof of 
Theorem~\ref{thmcop}, concerning the copolymer model, follows the 
same line of arguments ---in fact, the upper bound is 
significantly easier in this case--- so we only sketch the proofs and highlight the differences 
in Sections~\ref{sec4} and~\ref{sec5}. 

\begin{remark}\rm
The upper bound on $h_c^{{\rm cop}}(\lambda)$ in relation \eqref{eq:copthm} for the copolymer
is the same as the lower bound \eqref{eq:liminff}, which was established in \cite{cf:BdHO}
as an application of a quenched large deviations principle, developed 
by the authors and their collaborators. 
Here we present an alternative and self-contained proof, which is remarkably short
(see Section~\ref{sec5}).
\end{remark}

The proof of the lower bound on $h_c(\cdot)$ is based on a refinement of the ``fractional moment and 
coarse graining'' method, developed in \cite{cf:DGLT, cf:GLT2}, see~\cite{cf:G2} for an overview. The 
standard application of this moment method makes use of a change of measure, which via 
the use of H\"older's inequality gives rise to an energy-entropy balance. In 
\cite{cf:DGLT, cf:GLT2}, the entropy factor can be bounded by an arbitrary constant, while the 
energy factor can be made arbitrarily small. Our refinement requires  optimizing this 
energy-entropy balance, which is crucial in obtaining the precise constants. We also need a 
refinement in the coarse graining procedure. In the standard application, the polymer only needs 
to place a positive fraction of monomers
at the interface in each visited coarse-grained block, while 
in our case, we need to ensure that this positive fraction is in fact close to $1$. For this step, 
$\alpha>1$ plays a crucial role.

\smallskip

The upper bound on $h_c(\cdot)$ makes use of the following smoothing inequality
(the distinction between the pinning and copolymer, as usual, is simply
due to their different parametrization).

\begin{theorem}\label{smoothing} 
There exists a constant $\epsilon_0 > 0$ and a continuous map
$(\beta,\delta) \mapsto A_{\beta,\delta}$ from $(0,\epsilon_0) \times (-\epsilon_0,\epsilon_0)$
to $(0,\infty)$, depending only on the disorder distribution and
such that $\lim_{(\gb,\gd)\to(0,0)} A_{\gb,\gd} =1$, with the following properties:
\begin{itemize}
\item for the pinning model,
for every $0 < \gb < \epsilon_0$ and $|t| < \beta \epsilon_0$
\begin{equation} \label{eq:smooth}
	0\leq \f^{{\rm pin}}(\beta, h^{\rm pin}_c(\gb)+t) \,\le\, 
	\frac{1+\alpha}{2} A_{\gb,\frac{t}{\beta}} \,\frac{t^2}{\beta^2} \,;
\end{equation}
\item for the copolymer model, 
for every $0 < \lambda < \epsilon_0$ and $|\delta| < \epsilon_0$, 
\begin{equation} \label{eq:smoothcop}
	0\leq \f^{{\rm cop}}(\lambda, h^{\rm cop}_c(\gl)+\gd) \,\le\, \frac{1+\alpha}{2} 
	A_{\lambda,\gd}\,\gd^2 \,.
\end{equation}
\end{itemize}
\end{theorem}
\noindent
The smoothing inequality was first proved in \cite{cf:GT}, without the precision 
on the constant and under more restrictive assumptions on the disorder. 
In the case of Gaussian disorder, it appears in \cite{cf:G1} with the right constant $(1+\ga)/2$,
cf. Theorem~5.6 and Remark~5.7 therein.
The general statements we use here are proved in \cite{cf:Ca}. 
We remark that the precise (asymptotic) constant $(1+\ga)/2$ is crucial in obtaining the exact 
limits of $h_c^{{\rm cop}}(\gl)/\gl$ and $h_c^{{\rm pin}}(\gb)/\gb^2$. 

The idea to prove the upper bound is to couple the smoothing inequality with a
rough linear (but quantitative) lower 
bound on the free energies. More precisely, we prove that for every $c\in\bbR$,
\begin{equation}\label{eq:flowerbds}
\liminf_{\beta\downarrow 0} \frac{\f^{{\rm pin}}(\beta, c\beta^2)}{\beta^2} 
\geq \frac{1}{\mu}\bigg[c-\frac{1}{2\mu} \bigg]
\quad \text{and}
\quad
\liminf_{\gl\downarrow 0}\frac{\f^{{\rm cop}}(\gl,c\gl)}{\lambda^2}\geq  c -\frac{1}{2}.
\end{equation}
Remarkably, enforcing the compatibility of these inequalities with the corresponding 
smoothing inequalities \eqref{eq:smooth} and \eqref{eq:smoothcop} leads to the sharp upper bound 
on the critical curves.
What actually lies behind this compatibility condition is a rare stretch strategy. Let us try to describe 
it heuristically. We do so in the copolymer case, which is easier.

We start by decomposing $\mathbb{N}=\cup_{i=1}^\infty B_i$ into blocks of length $M$ and 
we search for such blocks
where the sample average of the disorder is about $-\lambda\delta$, that is 
$M^{-1}\sum_{n\in B_\cdot} \omega_n
\simeq -\lambda\delta$, where
$\delta$ is a fixed parameter. 
The probability of a block to have a sample average of that order is roughly
$\exp(-\lambda^2\delta^2 M/2)$ and the reciprocal of this probability will give the number 
of blocks that will separate the
atypical ones. Once these ``atypical blocks'' have been identified, 
we let the polymer jump from the end point of one such block to the
start point of the following. 
In view of \eqref{eq:ass}, the cost for this is
roughly $\exp(-(1+\alpha) \lambda^2\delta^2 M/2)$.
Once at the beginning of an atypical block, the contributions to the free energy of the copolymer is
\begin{eqnarray}\label{rarestretch}
\mathbb{E}\log \E\left[\prod_{j=1}^{\mathcal{N}_M} 
\frac{1}{2}\left(1+e^{ -2\lambda \sum_{n \in (\tau_{j-1}, \tau_j]} 
(\omega_n-\lambda\delta +h_a(\lambda)-h)} \right)\right],
\end{eqnarray}
where $\mathcal{N}_M$ is the number of excursions within the atypical block $B_\cdot$. 
Notice that in the above expectation we have integrated out the signs $\hat\varepsilon$ of the 
excursions of the path, while we have shifted the mean of the disorder to $-\lambda\delta$. 
Setting $h=c\lambda$, applying  Jensen's inequality and a Taylor expansion for small values 
of $\lambda$ gives that  \eqref{rarestretch} is bounded below by
$(c+\delta-\frac{1}{2})\lambda^2 M$, see \eqref{eq:flowerbds}. The energy-entropy balance 
gives the lower bound for the free energy
$$
e^{-\lambda^2 \delta^2 M/2} \, M
\left[ \bigg(c+\delta -\frac{1}{2}\bigg) -(1+\alpha)\frac{\delta^2}{2}\right]\lambda^2.
$$
Finally, optimizing over $\delta$, the term in square brackets
becomes $\big[c-\frac{\alpha}{2(1+\alpha)}\big]$, which leads to the sharp upper 
bound \eqref{eq:copthm} on the critical curve. 

Let us note that rare stretch strategies have been employed extensively in 
the study of pinning and copolymer models, cf. \cite[sections 6.3 and 5.4]{cf:G1} and 
\cite[section 5.1]{cf:G2} for instance. However, in most cases (in particular in the copolymer case) 
the strategy imposed on the copolymer after landing on an atypical block is to sample the whole 
disorder by staying exclusively in a single solvent. Our approach shows that the polymer follows rather 
more sophisticated strategies. In the case of a renewal with finite mean, this is captured by an 
averaging over the signs of the excursions and an optimization of the sample mean of the disorder in 
the targeted blocks. Apparently, this optimization makes the application of Jensen's inequality 
sharp. An analogous rare stretch strategy, but without optimizing over the disorder mean,
was used in \cite{cf:BGLT}.
\smallskip

Finally, a remark on notations. To ease the reading, we will drop the superscripts ${\rm pin}$ 
and ${\rm cop}$ from our notation for the free energy, partition function, and critical curve. 
This should not lead to confusion, since pinning and copolymer models are treated in 
separate sections. Moreover, we will refrain from using the integer parts, that is, instead 
of $\lfloor x \rfloor$ we simply write $x$. It will be clear from the context when 
the integer part of $x$ is used.

\section{On the Pinning Model: Lower Bound}
\label{sec2}

As already mentioned, we use the
``fractional moment and coarse graining'' method (see~\cite{cf:G2}),
but with several crucial refinements, that we now explain.

\subsection{The general strategy} \label{S:warmup}

To obtain a lower bound
on the critical curve $h_c(\beta)$, it suffices to prove $\f(\beta, h)=0$ for suitably chosen $h$ 
as a function of $\beta$. This is further reduced to showing that for some $\zeta \in (0,1)$, we have
\begin{equation} \label{eq:ttopprove1}
	\liminf_{N\to\infty}
	\, \bbE\big[ \big( Z_{N,\beta,h}^{\rc, \omega} \big)^\zeta \big] \,<\, \infty \,.
\end{equation}
Indeed, note that
\begin{equation*}
\begin{aligned}
	\f(\beta, h) =\! \liminf_{N\to\infty}
	\frac{1}{N} \bbE \big[ \log Z_{N,\beta,h}^{\rc, \omega} \big] &\,=\, \liminf_{N\to\infty}
	\frac{1}{N\zeta} \bbE \big[ \log \big( Z_{N,\beta,h}^{\rc, \omega} \big)^\zeta \big] \\
	&\,\le \, \liminf_{N\to\infty}
	\frac{1}{N\zeta} \log \bbE \big[ \big( Z_{N,\beta,h}^{\rc, \omega} \big)^\zeta \big] = 0.
\end{aligned}
\end{equation*}
To obtain \eqref{eq:ttopprove1}, we employ a coarse-graining scheme.
The idea is to divide the system into (large) finite blocks, each one being of size $k$,
the \emph{correlation length} of the annealed model, proportional to $1/\gb^2$.
We estimate the partition functions on different blocks separately, and then ``glue'' these
estimates together through a coarse-graining procedure.

\smallskip

We first estimate the partition function of a system of size $k$,
the coarse-graining length scale.
Let $\tilde\bbP_{-\delta,k}$ denote the law under which $\{\omega_i\}_{1 \le i \le k}$
are i.i.d.\ with density
\begin{equation} \label{eq:tildeP}
	\frac{\dd\tilde\bbP_{-\delta,k}}{\dd\bbP} \,=\,
	\prod_{i=1}^k e^{-\delta \omega_i - \gL(-\gd)},
\end{equation}
which is an exponential tilting of the law of $\{\omega_i\}_{1\leq i\leq k}$. 
We then apply the standard change of measure trick: by H\"older's inequality, 
for any $\zeta \in (0,1)$
\begin{equation} \label{eq:trick}
	\bbE\big[ \big( Z_{k,\beta,h}^{\rc,\omega} \big)^\zeta \big]
	= \tilde\bbE_{-\delta,k}\bigg[
	\big( Z_{k,\beta,h}^{\rc,\omega} \big)^\zeta \, \frac{\dd\bbP}{\dd\tilde\bbP_{-\delta,k}} \bigg]
	\,\le\, \tilde\bbE_{-\delta,k} \big[ Z_{k,\beta,h}^{\rc,\omega} \big]^\zeta
	\, \tilde\bbE_{-\delta,k} \bigg[ \bigg( \frac{\dd\bbP}{\dd\tilde\bbP_{-\delta,k}}
	\bigg)^{\frac{1}{1-\zeta}} \bigg]^{1-\zeta} \,.
\end{equation}
The second factor is easily computed:
\begin{eqnarray}\label{eq:entrp}
	\tilde\bbE_{-\delta,k} \bigg[ \bigg( \frac{\dd\bbP}{\dd\tilde\bbP_{-\delta,k}}
	\bigg)^{\frac{1}{1-\zeta}} \bigg]^{1-\zeta} &\,=\,&
	\bbE \bigg[ \bigg( \frac{\dd\bbP}{\dd\tilde\bbP_{-\delta,k}}
	\bigg)^{\frac{\zeta}{1-\zeta}} \bigg]^{1-\zeta} \,=\,
	\bbE \Bigg[
	\prod_{i=1}^k e^{\frac{\zeta}{1-\zeta} \delta \omega_i +
	\frac{\zeta}{1-\zeta}\gL(-\gd)} \Bigg]^{1-\zeta}\nonumber\\
	&\,=\,&
	 e^{ (1-\zeta)k \left[ \gL(\frac{\zeta}{1-\zeta}\gd)+\frac{\zeta}{1-\zeta}\gL(-\gd) \right] } \,.
\end{eqnarray}
Using \eqref{eq:Zpinc} and recalling
\eqref{eq:hannpin}, the first factor in \eqref{eq:trick} can also be computed:
\begin{equation}\label{eq:tildeZ}
	\tilde\bbE_{-\delta,k} \big[ Z_{k,\beta,h}^{\rc,\omega} \big]
	= \E\Big[ e^{(\gL(\beta-\delta)-\gL(\gb)-\gL(-\gd)+h) \, |\tau \cap [1,k]|}
	\,1_{\{k \in \tau\}} \Big] \,.
\end{equation}
Since $\tau_{|\tau \cap [1,k]|} \le k \le
\tau_{|\tau \cap [1,k]| + 1}$, it follows by the strong law of large numbers that
\begin{equation} \label{eq:lkj}
	\frac{|\tau \cap [1,k]|}{k} \,\xrightarrow[\ k \to \infty\ ]\, \frac{1}{\mu} \qquad 
	\text{$\P$-a.s.}\,.
\end{equation}
Since $\P(k \in \tau) \to \frac{1}{\mu} > 0$ as $k\to\infty$, by the renewal theorem,
we have the convergence in distribution
\begin{equation} \label{eq:weakco}
	\frac{|\tau \cap [1,k]|}{k} \,\xrightarrow[\ k \to \infty\ ]{\rm d}\,
	\frac{1}{\mu} \qquad 
	\text{under } \P(\,\cdot\,|\,k \in \tau)\,.
\end{equation}

We now parametrize everything in terms of $\beta$. Let us set
\begin{equation*}
	h_\beta = c \beta^2 \,, \quad \delta_\beta = a \beta \,, \quad
	k_\beta = \frac{t}{\beta^2} \,, \qquad \text{for } c,a,t \in (0,\infty) \text{ with }
	a > c \,.
\end{equation*}
As $\beta \downarrow 0$, we have $k_\beta \to \infty$ and,
recalling \eqref{eq:asLambda},
$$
( \gL(\beta-\delta_\gb)-\gL(\gb)-\gL(-\gd_\gb)+h_\gb) \frac{k_\beta}{\mu}
\sim (h_\beta - \beta\delta_\beta) \frac{k_\beta}{\mu}
\;\longrightarrow\; (c-a)\frac{t}{\mu} \;\in\; (-\infty, 0) \,.
$$
Together with the fact that $\P(k \in \tau) \to \frac{1}{\mu}$ as $k\to\infty$, 
it follows from \eqref{eq:tildeZ} and \eqref{eq:weakco} that
\begin{equation} \label{eq:csdf}
	\lim_{\beta \downarrow 0} \, \tilde\bbE_{-\delta_\beta, k_\beta}
	\big[ Z_{k_\beta,\, \beta,\, h_\beta}^{\rc,\omega} \big]
	\,=\, \frac{1}{\mu} \, e^{(c-a) \frac{t}{\mu} } \,,
\end{equation}
hence, by \eqref{eq:trick} and \eqref{eq:entrp},
\begin{equation} \label{eq:optquas}
	\limsup_{\beta\downarrow 0}
	\bbE\big[ \big( Z_{k_\beta,\beta,h_\beta}^{\rc,\omega} \big)^\zeta \big]
	\,\le\, \frac{1}{\mu^\zeta} \,
	e^{\frac{\zeta}{\mu} (c-a) t } \, e^{\frac{a^2}{2} \frac{\zeta}{1-\zeta} t}
	\qquad \forall\, a > c \,.
\end{equation}
Note that the exponent is a polynomial of second degree in $a$:
optimizing over $a$ yields
\begin{equation}\label{eq:aopt}
	a = \frac{1-\zeta}{\mu} \,,
\end{equation}
and one gets the basic estimate
\begin{equation}\label{eq:fracest}
	\limsup_{\beta\downarrow 0}
	\bbE\big[ \big( Z_{k_\beta,\beta,h_\beta}^{\rc,\omega} \big)^\zeta \big]
	\,\le\, \frac{1}{\mu^\zeta}\,
	\exp\bigg\{ \frac{\zeta}{\mu} \bigg( c-\frac{1-\zeta}{2\mu} \bigg) t \bigg\}.
\end{equation}
To feed this estimate into the coarse graining scheme and obtain \eqref{eq:ttopprove1}, 
we need to make the right hand side arbitrarily small. 
This can be accomplished by choosing $t$ large enough,
\emph{provided $c< (1-\zeta)/(2\mu)$}, or equivalently, 
$h/ \beta^2<(1-\zeta)/(2\mu)$. As we will see later, the coarse graining scheme works only 
if $\zeta>1/(1+\alpha)$, which leads to
$h/\beta^2 < \alpha/(2(1+\alpha)\mu)$ and thus to the sharp lower bound
\eqref{eq:CSZ+} on $h_c(\beta)$.

\smallskip

Note that the bound \eqref{eq:fracest} is derived via a \emph{subtle balance between 
the cost of changing
the measure and the annealed partition function under the changed measure}, i.e., the two
factors in the right hand side \eqref{eq:trick}. 
This is in contrast to \cite[Proposition~7.1]{cf:G2}, where:
\begin{itemize}
\item the cost of changing
the measure is only required to be an arbitrary fixed
constant for each coarse graining block;
\item the annealed partition function under the changed measure is small over any interval whose
length exceeds a $\delta$-proportion of the coarse graining block. 
\end{itemize}
In our case, the cost of changing the measure needs to be estimated sharply;
furthermore, in order to balance this cost and get \eqref{eq:fracest}, 
\emph{we need to average the partition function under the
changed measure over an interval whose length is close to the full coarse graining block}.
Fortunately such configurations can be shown to give the dominant contribution in the case $\alpha>1$.
We stress that this is \emph{not} the case when $\alpha < 1$.

\subsection{Proof of the lower bound}\label{S:coarsegraining}
We divide the proof into several steps.

\medskip

\noindent
{\bf STEP 1.} Let us first set up the proper framework. To prove the lower bound for $h_c(\beta)$ in \eqref{eq:CSZ+}, we show that
for every $\epsilon > 0$ small enough there exists $\beta_0 = \beta_0(\epsilon) \in (0,\infty)$
such that for every $\beta \in (0, \beta_0)$ we have
\begin{equation} \label{eq:fceps}
	\f\big( \beta, \, { c}_\epsilon\, \beta^2
	\big) \,=\, 0 \,, \qquad \text{where} \qquad
	{ c}_\epsilon \,:=\,
	(1-\epsilon) \frac{\alpha}{1+\alpha} \frac{1}{2 \mu} \,.
\end{equation}
As explained in Section~\ref{S:warmup}, it suffices to show that there exists
$\zeta = \zeta_\epsilon \in (0,1)$ such that
\begin{equation} \label{eq:ttopprove}
	\liminf_{N\to\infty} \bbE\big[ \big( Z_{N,\beta, c_\epsilon \beta^2}^{\rc, \omega} 
	\big)^\zeta \big] < \infty.
\end{equation}


\smallskip

Henceforth let $\epsilon \in (0,1)$ be fixed. We then set
\begin{equation} \label{eq:zetaeps}
	\zeta_\epsilon \,:=\, \frac{1}{1+\alpha} \,+\, \frac{\epsilon}{2}\frac{\alpha}{1+\alpha}
	\,\in\, \big( \tfrac{1}{1+\alpha}, 1 \big) \,.
\end{equation}
For later convenience,
we assume that $\epsilon$ is small enough so that
$(1+\alpha - \frac{\epsilon}{2}) \zeta_\epsilon > 1$ and also $\alpha - \epsilon/2 > 1$
(which is possible, since $\alpha > 1$).
We set the coarse graining length scale to be
\begin{equation} \label{eq:kbetaeps}
	k= k_{\beta, \epsilon} \,=\, t_\epsilon/\beta^2
\end{equation}
for some $t_\epsilon \in (0,\infty)$, which \emph{depends only
on $\epsilon$} and will be fixed at the end of the proof.

\smallskip

Recall from \eqref{eq:tildeP} the exponentially tilted law $\tilde\bbP := \tilde\bbP_{-\delta,k}$.
We will use it with $k=k_{\beta,\epsilon}$ and
\begin{equation} \label{eq:deltabetaeps}
	\delta =\delta_{\beta,\epsilon} \,:=\, a_\epsilon \beta
	\,:=\, \frac{1-\zeta_\epsilon}{\mu} \beta,
\end{equation}
where the choice of $a_\epsilon$ is (a posteriori) optimal, recall \eqref{eq:aopt}.
Note that $c_\epsilon \le (1-\zeta_\epsilon) \frac{1}{2\mu}$
by \eqref{eq:fceps} and \eqref{eq:zetaeps},
hence $a_\epsilon \ge 2 c_\epsilon$ by \eqref{eq:deltabetaeps}. In particular, we stress that
\begin{equation}\label{eq:a>c}
	a_\epsilon > c_\epsilon \,,
\end{equation}
a relation that will be used several times in the sequel.

Now we recall the crucial relation \eqref{eq:csdf}, which can be rewritten in our current setting as
\begin{equation} \label{eq:csdf+}
	\lim_{\beta \downarrow 0} \tilde\bbE
	\big[ Z_{t_\epsilon \beta^{-2},\, \beta,\, c_\epsilon \beta^2}^{\rc,\omega} \big]
	\,=\, \frac{1}{\mu} \, e^{-\frac{1}{\mu} (a_\epsilon - c_\epsilon) t_\epsilon } \,.
\end{equation}
The convergence \eqref{eq:csdf+} is actually \emph{uniform}
when $t_\epsilon$ varies in a compact subset of $(0,\infty)$. In 
particular, there exists $\beta_1(\epsilon) > 0$ such that
for all $\beta \in (0, \beta_1(\epsilon))$ and all $n\in\N$ 
with $(1-\epsilon/5) t_\epsilon \beta^{-2} \leq n\leq t_\epsilon\beta^{-2}$, we have
\begin{equation} \label{eq:intesti}
\frac{1-\epsilon}{\mu} \, e^{-\frac{1}{\mu} (a_\epsilon - c_\epsilon) \beta^2 n}
	\,\le\, \tilde\bbE \big[ Z_{n,\, \beta,\, c_\epsilon \beta^2}^{\rc,\omega} \big]
	\,\le\, \frac{1+\epsilon}{\mu} \, e^{-\frac{1}{\mu} (a_\epsilon - c_\epsilon) \beta^2 n}.
\end{equation}
This uniform bound follows from the convergence in \eqref{eq:weakco},
because the functions $x \mapsto e^{-C x}$ are uniformly
bounded and uniformly Lipschitz, if $C$ ranges over a bounded set.
Note that the upper and lower bounds in \eqref{eq:intesti}
are bounded away from $0$ and $\infty$ (for a fixed $\epsilon > 0$), 
because $(1-\tfrac{\epsilon}{5}) t_\epsilon \le \beta^2 n \le t_\epsilon$.

\medskip

\noindent
{\bf STEP 2.} We now develop the coarse graining scheme. The system size $N$ will be a multiple of 
the coarse graining length scale: $N = m k=m t_\epsilon \beta^{-2}$ for some $m\in\N$. We 
then partition $\{1,\ldots, N\}$ into $m$ blocks $B_1, \ldots, B_m$ of size 
$k=t_\epsilon \beta^{-2}$, defined by
\begin{equation*}
	B_i \,:=\, \big\{(i-1)k + 1, \ldots, i k\big\} \,\subseteq\, \{1,\ldots,N\} \,,
\end{equation*}
so that the macroscopic (coarse-grained) ``configuration space'' is $\{1,\ldots,m\}$.
A macroscopic configuration is a subset $J \subseteq \{1,\ldots,m\}$. By a decomposition 
according to which blocks are visited by the renewal process (we call these blocks occupied), 
we can then write
\begin{equation*}
	Z_{N,\beta, c_\epsilon \beta^2}^{\rc, \omega} \,=\, \sum_{J \subseteq \{1,\ldots, m\}:
	\ m \in J} \hat Z_J
\end{equation*}
where for $J = \{j_1, \ldots, j_\ell\}$, with $1 \leq  j_1 < j_{2} < \ldots < j_\ell = m$ 
and $\ell = |J|$,
\begin{equation}\label{ZJpin}
\begin{split}
	\hat Z_J  \,:=\, \sumtwo{d_1, f_1 \in B_{j_1}}{d_1 \le f_1}
	\ldots \sumtwo{d_{\ell-1}, f_{\ell-1} \in B_{j_{\ell-1}}}{d_{\ell - 1} \le f_{\ell - 1}}
	\sum_{d_\ell \in B_{j_\ell} = B_m}
	\Bigg( \prod_{i=1}^\ell K(d_i - f_{i-1}) z_{d_i} Z_{d_i, f_i} \Bigg) \,,
\end{split}
\end{equation}
where we set $f_0 := 0$ and $f_\ell := N = m k$, and for all $d<f\in\N$,
\begin{equation}\label{zdf}
	z_d \,:=\, e^{\beta \omega_d - \Lambda(\gb) + c_\epsilon \beta^2} \,,
	\qquad
	Z_{d,f} \,:=\, Z_{f-d,\beta,c_\epsilon \beta^2}^{\rc, \theta^d\omega} \,,
\end{equation}
with $\theta^d\omega:= \{(\theta^d \omega)_n\}_{n\in\N} = \{\omega_{n+d}\}_{n\in\N}$ defined as a 
shift of the disorder $\omega$. Since $\zeta_{\epsilon}<1$, one has that
$(a+b)^{\zeta_\epsilon} \le a^{\zeta_\epsilon} + b^{\zeta_\epsilon}$, for all $a,b\ge 0$, and 
consequently
\begin{equation} \label{eq:fracmom}
	\bbE \big[ \big( Z_{N,\beta, c_\epsilon \beta^2}^{\rc, \omega} \big)^{\zeta_\epsilon} \big]
	\,\le\, \sum_{J \subseteq \{1,\ldots, m\}:
	\ m \in J} \bbE \big[ \big( \hat Z_J \big)^{\zeta_\epsilon} \big].
\end{equation}

To bound $\bbE \big[ \big( \hat Z_J \big)^{\zeta_\epsilon} \big]$, we apply the change of measure as in \eqref{eq:tildeP}. Let $\tilde\bbP_{J}$ be the law of the disorder obtained from $\bbP$, where independently for each $n \in \bigcup_{i \in J} B_i$, the law of $\omega_n$ is tilted with density $e^{-\delta \omega_n-\Lambda(-\delta)}$, with $\delta=a_\epsilon\beta$ as chosen in \eqref{eq:deltabetaeps}. Then by the same argument as in \eqref{eq:trick}, we have
\begin{equation} \label{eq:Holder}
	\bbE \big[ \big( \hat Z_J \big)^{\zeta_\epsilon} \big] \,\le\,
	\tilde\bbE_J \big[ \hat Z_J \big]^{\zeta_\epsilon} \,
	\tilde\bbE_{J} \bigg[ \bigg( \frac{\dd\bbP}{\dd\tilde\bbP_{J}}
	\bigg)^{\frac{1}{1-{\zeta_\epsilon}}} \bigg]^{1-{\zeta_\epsilon}} \,.
\end{equation}
To bound the second factor, note that for $J = \{j_1, \ldots, j_\ell\}$, the same calculation as in \eqref{eq:entrp} gives (recall that $k=t_\epsilon \beta^{-2}$ and $\delta=a_\epsilon \beta$)
\begin{align}
	\nonumber
	\tilde\bbE_{J} \bigg[ \bigg( \frac{\dd\bbP}{\dd\tilde\bbP_{J}}
	\bigg)^{\frac{1}{1-{\zeta_\epsilon}}} \bigg]^{1-{\zeta_\epsilon}} & \,=\,
	\bbE \bigg[ \bigg( \frac{\dd\bbP}{\dd\tilde\bbP_{J}}
	\bigg)^{\frac{{\zeta_\epsilon}}{1-{\zeta_\epsilon}}} \bigg]^{1-{\zeta_\epsilon}} \,=\,
	\bbE \Bigg[ \prod_{i=1}^\ell \prod_{n \in B_{j_i}}
	e^{\frac{{\zeta_\epsilon}}{1-{\zeta_\epsilon}} \delta \omega_n +
	\frac{{\zeta_\epsilon}}{1-{\zeta_\epsilon}}
	\Lambda(-\gd)} \Bigg]^{1-{\zeta_\epsilon}} \\
	\label{eq:RN}
	& \,=\, e^{  (1-\zeta_\epsilon) 
	\left[ \gL(\frac{\zeta_\epsilon}{1-\zeta_\epsilon}\gd)
	+\frac{\zeta_\epsilon}{1-\zeta_\epsilon}\gL(-\gd) \right]
	 k \ell}
	 \,<\, e^{(1+\frac{\epsilon}{20})\frac{a_{\epsilon}^2}{2}
	\frac{{\zeta_\epsilon}}{1-{\zeta_\epsilon}} t_{\epsilon}|J|}\,,
\end{align}
where the last inequality holds, by \eqref{eq:asLambda}, for $\beta$ small enough,
say $\beta \in (0, \beta_2(\epsilon))$, for some $\beta_2(\epsilon) > 0$.
To bound the first factor in \eqref{eq:Holder},
recall that $a_\epsilon > c_\epsilon$ and note that, for every $d \in \bigcup_{i\in J} B_i$, 
for $\beta < \beta_2(\epsilon)$ we have
\begin{equation}\label{zn}
	\tilde\bbE_J[z_d] \,=\, e^{\gL(\beta-\delta)-\gL(\gb)-\gL(-\gd)+c_{\epsilon}\gb^2}
	\,=\, e^{-(a_\epsilon - c_\epsilon) \beta^2+o(\beta^2)} \,<\, 1 \,,
\end{equation}
provided $\beta_2(\epsilon)$ is chosen small enough. 
Furthermore, for $d, f \in B_i$ for any $i \in J$,
\begin{equation} \label{eq:tildeu}
	\tilde\bbE_J[ Z_{d,f}] \,=\, \tilde u(f-d) \,, \qquad \text{where} \qquad
	\tilde u(n) \,:=\, \tilde\bbE \big[
	Z_{n,\, \beta,\, c_\epsilon \beta^2}^{\rc, \omega} \big] \,.
\end{equation}
Therefore from \eqref{ZJpin}, we obtain
\begin{equation} \label{eq:towork}
	\tilde\bbE_J \big[ \hat Z_J \big] \,\le\,
	\sumtwo{d_1, f_1 \in B_{j_1}}{d_1 \le f_1}
	\ldots \sumtwo{d_{\ell-1}, f_{\ell-1} \in B_{j_{\ell-1}}}{d_{\ell - 1} \le f_{\ell - 1}}
	\sum_{d_\ell \in B_{j_\ell} = B_m}
	\Bigg( \prod_{i=1}^\ell K(d_i - f_{i-1}) \tilde u(f_i - d_i) \Bigg) \,.
\end{equation}
This expression is nice because it would
be the probability of a renewal event, if $\tilde u$ were replaced by
\begin{equation*}
	u(n) \,:=\, \P(n \in \tau).
\end{equation*}
In fact, we will replace $\tilde u(\cdot)$ by a {\it small} multiple of $u(\cdot)$.
First we need some estimates.
Since $u(n)\to 1/\mu$ as $n\to\infty$ by the renewal theorem, 
we can choose $C_1 \in (0,\infty)$ such that
\begin{equation} \label{eq:univu}
\frac{1}{C_1 \mu} \,\le\, u(n) \,\le\, \frac{C_1}{\mu} \qquad \forall\, n \in\N\,.
\end{equation}
Furthermore, by relation \eqref{eq:ass}
and the fact that slowly varying functions are asymptotically dominated by any
polynomial, it follows that there exists $C_{2,\epsilon} \in (0,\infty)$ be such that
\begin{equation} \label{eq:univK}
	K(n) \,\le\, \frac{C_{2,\epsilon}}{n^{1+\alpha-\frac{\epsilon}{2}}} \,, \qquad
	\forall n \in \N\,.
\end{equation}
Finally, since $0 \le |\tau \cap [1,n]| \le k$
for every $n \in \{1, \ldots, k=t_\epsilon\beta^{-2}\}$,
recalling \eqref{eq:tildeZ} we obtain
\begin{equation} \label{eq:trivb}
\begin{split}
	e^{-2(a_\epsilon - c_\epsilon) t_\epsilon} u(n)
	\,\le\, \tilde u(n) \,=\, \tilde\bbE \big[ Z_{n,\, \beta,\, c_\epsilon \beta^2}^{\rc,\omega} \big]
	& \,=\, \E\Big[ e^{\big(-(a_\epsilon - c_\epsilon) \beta^2+o(\gb^2)\big) \, |\tau \cap [1,n]|}
	\,\ind_{\{n \in \tau\}} \Big]\\
	& \,\le\, u(n) \,\leq\, 1 \,,
\end{split}
\end{equation}
for $\beta \in (0, \beta_2(\epsilon))$, because $a_\epsilon > c_\epsilon$ (recall \eqref{zn}).

\medskip

\noindent
{\bf STEP 3.} We now replace $\tilde u(\cdot)$ in \eqref{eq:towork}
by a suitable small multiple of $u(\cdot)$. However, this is only possible
for occupied blocks that are surrounded by occupied blocks. The blocks with
unoccupied neighboring blocks have to be dealt with in a different way.

Let us be more precise. Fix $i$ such that both $j_i \in J$ and $j_i-1 \in J\cup\{0\}$.
Then we claim that the terms in \eqref{eq:towork} with $|d_i - f_{i-1}| \le \frac{\epsilon}{10} k$
give the main contribution. Indeed, setting
$\overline f_{i-1} = (i-1)k$,
$\overline d_{i} = (i-1)k + 1$,
\begin{equation}
\label{eq:comparShort}
	\tilde u(\overline f_{i-1} - \overline d_{i-1}) \,
	K(\overline d_i - \overline f_{i-1}) \, \tilde u(\overline f_i - \overline d_i)
	\,\ge\, \big( e^{-2(a_\epsilon - c_\epsilon) t_\epsilon} \frac{1}{C_1\mu}\big)^2 \, K(1) \,,
\end{equation}
where we used the lower bounds in \eqref{eq:trivb} and \eqref{eq:univu}. 
Using instead the upper bound in \eqref{eq:trivb} that $\tilde u(\cdot) \leq 1$,
together with \eqref{eq:univK}, yields
\begin{equation}
\label{eq:comparLong}
	\sumtwo{f_{i-1} \in B_{j_{i-1}}, \ d_i \in B_{j_i}}{|d_i - f_{i-1}| >
	\frac{\epsilon}{10} k}
	\tilde u(f_{i-1} - d_{i-1}) \, K(d_i - f_{i-1}) \, \tilde u(f_i - d_i)
	\,\le\, \frac{k^2 \, C_{2,\epsilon}}{(\frac{\epsilon}{10} k)^{1+\alpha-\frac{\epsilon}{2}}}
	\,=\, \frac{10^{1+\alpha} \, C_{2,\epsilon}}{\epsilon^{1+\alpha} \,
	k^{\alpha - \frac{\epsilon}{2} -1}} \,.
\end{equation}
Recall that $\epsilon$ is chosen small enough, so that
$\alpha-\frac{\epsilon}{2} > 1$,
and $k = t_{\epsilon}\beta^{-2} \to \infty$ as $\beta \downarrow 0$.
Therefore, we can find $\beta_3(\epsilon) \in (0,\infty)$
such that for every $\beta \in (0, \beta_3(\epsilon))$, the contribution of the 
terms in \eqref{eq:towork} with $|d_i - f_{i-1}| > \frac{\epsilon}{10} k$ is 
smaller than the contribution of the terms with $|d_i - f_{i-1}| \le \frac{\epsilon}{10} k$ 
(comparing \eqref{eq:comparShort} and \eqref{eq:comparLong}).

To summarize: when $\beta< \beta_3(\epsilon)$, the right hand side of \eqref{eq:towork} 
can be bounded from above by restricting the sum
to $|d_i - f_{i-1}| \le \frac{\epsilon}{10} k$ for every $i$ such that both $j_i \in J$ 
and $j_i-1 \in J$, provided one introduces a multiplicative factor of $2$ for each such $i$. 

Let us now set
\begin{equation*}
	\overset{\circ}{J} \,:=\, \big\{ j\in J:\ j - 1 \in J\cup\{0\} \text{ and } 
	j+1 \in J\cup\{m+1\} \big\} \,.
\end{equation*}
If $j \in \overset{\circ}{J}$, say $j = j_i$ for some $1\leq i\leq |J|=l$, then we have 
restricted the summation in \eqref{eq:towork} to both $|d_i - f_{i-1}| \le \frac{\epsilon}{10} k$ 
and $|d_{i+1} - f_{i}| \le \frac{\epsilon}{10} k$, which yields
$f_i - d_i > (1-\frac{\epsilon}{5}) k$. Recalling \eqref{eq:tildeu}, \eqref{eq:intesti} 
and \eqref{eq:univu}, we can then bound
\begin{equation} \label{eq:Deps}
	\tilde u(f_i - d_i) \,\le\, D_\epsilon u(f_i - d_i) \,, \qquad
	\text{where} \qquad
	D_\epsilon \,:=\, C_1 (1+\epsilon)\, e^{-\frac{1}{\mu} (a_\epsilon - c_\epsilon)
	(1-\tfrac{\epsilon}{5}) t_\epsilon} \,.
\end{equation}
This is the crucial replacement, after which we can remove the 
restrictions $|d_i - f_{i-1}| \le \frac{\epsilon}{10} k$
in \eqref{eq:towork} to get an upper bound. 
(Note that $D_\epsilon$
can be made arbitrarily small by choosing $t_\epsilon$ large, because $a_\epsilon>c_\epsilon$,
so we may assume henceforth that $D_\epsilon < 1$.)

It only remains to deal with the terms $j \in J \setminus \overset{\circ}{J}$,
i.e. the occupied blocks that have at least one neighboring block which is unoccupied.
For these blocks, we replace $\tilde u(f_i - d_i)$ by $u(f_i - d_i)$, 
thanks to \eqref{eq:trivb}. 
Gathering the above considerations, we can upgrade \eqref{eq:towork} to
\begin{equation} \label{eq:towork+}
	\tilde\bbE_J \big[ \hat Z_J \big] \le 2^{|J|} 
	(D_\epsilon)^{|\overset{\circ}{J}|} \!\!\!\!\sumtwo{d_1, f_1 \in B_{j_1}}{d_1 \le f_1}
	\!\!\cdots\!\!\!\!\!\! 
	\sumtwo{d_{\ell-1}, f_{\ell-1} \in B_{j_{\ell-1}}}{d_{\ell - 1} \le f_{\ell - 1}}
	\sum_{d_\ell \in B_{j_\ell} = B_m}
	\Bigg( \prod_{i=1}^\ell K(d_i - f_{i-1}) u(f_i - d_i) \Bigg) \,,
\end{equation}
where we note that the summation is now the probability of a renewal event. 

\medskip

\noindent
{\bf STEP 4.}  We now deal with the gaps between occupied blocks. 
Let $i \in \{1,\ldots, \ell\}$ be such that
$j_i \in J$ but $j_i - 1 \not\in J\cup\{0\}$, that is $j_i - j_{i-1} \ge 2$. Since
$d_i \in B_{j_i}$ and $f_{i-1} \in B_{j_{i-1}}$, we have
$d_i - f_{i-1} \ge (j_i - j_{i-1} - 1) k$.
Then it follows from \eqref{eq:univK} that
\begin{equation} \label{eq:boundKaa}
	K(d_i - f_{i-1}) \,\le\, \frac{C_{2,\epsilon}}{k^{1+\alpha-\frac{\epsilon}{2}}} \,
	\frac{1}{(j_i - j_{i-1} - 1)^{1+\alpha-\frac{\epsilon}{2}}}
	\,\le\, \frac{2^{1+\alpha} \, C_{2,\epsilon}}{k^{1+\alpha-\frac{\epsilon}{2}}} \,
	\frac{1}{(j_i - j_{i-1})^{1+\alpha-\frac{\epsilon}{2}}} \,,
\end{equation}
where the last inequality holds because $n-1 \ge \frac{n}{2}$ for $n\ge 2$.
Furthermore, by \eqref{eq:univu},
\begin{equation} \label{eq:bounduaa}
	u(f_{i-1} - d_{i-1}) \,\le\, C_1^2 \, u(\overline{f}_{i-1} - d_{i-1}) \,, \qquad
	u(f_i - d_i) \,\le\, C_1^2 \, u(f_i - \overline{d}_i) \,,
\end{equation}
where we recall that $\overline f_{i-1} = (i-1)k$ and
$\overline d_{i} = (i-1)k + 1$ denote respectively the last point of the block
$B_{i-1}$ and the first point of the block $B_i$.

We can now insert the bounds \eqref{eq:boundKaa}, \eqref{eq:bounduaa} 
into \eqref{eq:towork+}, starting with the smallest
$i \in \{1, \ldots, \ell\}$ such that $j_i - j_{i-1} \ge 2$ (if any),
and then proceeding in increasing order.
(When there are two consecutive gaps, that is, when both $j_i - j_{i-1} \ge 2$ 
and $j_{i-1} - j_{i-2} \ge 2$, the first bound in \eqref{eq:bounduaa}
becomes $u(f_{i-1} - \overline{d}_{i-1}) \,\le\, C_1^2 \, 
u(\overline{f}_{i-1} - \overline{d}_{i-1})$,
because we have already replaced $d_{i-1}$ by $\overline{d}_{i-1}$ in the previous step.)
In this way, we eliminate all the terms in \eqref{eq:towork+} 
that depend on $f_{i-1}$ and $d_i$, and the double sum over $f_{i-1}$ and $d_i$ can be removed 
by introducing a multiplicative factor $k^2$.

Having eliminated the gaps, we are left with ``clusters of consecutive occupied blocks'':
more precisely, the surviving sums in \eqref{eq:towork+} are those
over variables $f_{j-1}, d_j$ with $a < j \le b$,
for every maximal interval $\{a, \ldots, b\} \subseteq \{1, \ldots, \ell\}$. These
sums can be factorized and each such interval gives a contribution
equal to the probability (hence bounded by $1$) that the renewal process visits a 
cluster of consecutive occupied blocks. Therefore
\begin{equation} \label{eq:uff1}
\begin{split}
	\tilde\bbE_J \big[ & \hat Z_J \big] \,\le\, 2^{|J|} \,
	(D_\epsilon)^{|\overset{\circ}{J}|} \,
	\prod_{i \in \{1,\ldots,\ell\}: \ j_i - j_{i-1} \ge 2}
	C_1^4 \, k^2 \, \frac{2^{1+\alpha} \, C_{2,\epsilon}}{k^{1+\alpha-\frac{\epsilon}{2}}} \,
	\frac{1}{(j_i - j_{i-1})^{1+\alpha-\frac{\epsilon}{2}}} \\
	& \,=\, 2^{|J|} \,
	(D_\epsilon)^{|\overset{\circ}{J}|} \,
	\bigg( \frac{2^{1+\alpha} \, C_1^4 \, C_{2,\epsilon}}{k^{\alpha-\frac{\epsilon}{2}-1}}
	\bigg)^{|\{ i \in \{1,\ldots,\ell\}: \ j_i - j_{i-1} \ge 2 \}|}
	\prod_{i \in \{1,\ldots,\ell\}}
	\frac{1}{(j_i - j_{i-1})^{1+\alpha - \frac{\epsilon}{2}}} \,.
\end{split}
\end{equation}

Next observe that
\begin{equation} \label{eq:uff2}
	|J \setminus \overset{\circ}{J}| \,=\,
	|J| - |\overset{\circ}{J}| \,\le\, 2 \,
	|\{ i \in \{1,\ldots,\ell\}: \ j_i - j_{i-1} \ge 2 \}| \,,
\end{equation}
since each point in $J\setminus \overset{\circ}{J}$ is either the starting point or ending point 
of a gap,
i.e., a pair $\{j_{i-1}, j_i\}$ with $j_i-j_{i-1}\geq 2$. 
Since $\alpha-\frac{\epsilon}{2}>1$, by our choice of $\epsilon$, 
and $k=t_{\epsilon} \beta^{-2} \to \infty$
as $\beta \downarrow 0$, there exists $\beta_4(\epsilon) \in (0,\infty)$ such that for 
every $\beta\in (0, \beta_4(\epsilon))$, we have
\begin{equation*}
	\frac{2^{1+\alpha} \, C_1^4 \, C_{2,\epsilon}}{k^{\alpha-\frac{\epsilon}{2}-1}}
	\,\le\, (D_\epsilon)^2 \,.
\end{equation*}
Since $D_\epsilon < 1$, it follows from \eqref{eq:uff1} and \eqref{eq:uff2} that
\begin{equation} \label{eq:towork++}
	\tilde\bbE_J \big[ \hat Z_J \big] \,\le\, (2 D_\epsilon)^{|J|} \,
	\prod_{i \in \{1,\ldots,\ell\}}
	\frac{1}{(j_i - j_{i-1})^{1+\alpha - \frac{\epsilon}{2}}} \,.
\end{equation}

\medskip

\noindent
{\bf STEP 5.} We now conclude the proof. Looking back at \eqref{eq:fracmom}, \eqref{eq:Holder},
\eqref{eq:RN} and \eqref{eq:towork++}, we can write
\begin{equation} \label{eq:pin}
	\bbE \big[ \big( Z_{N,\beta, c_\epsilon \beta^2}^{\rc, \omega} \big)^{\zeta_\epsilon} \big]
	\,\le\, \sum_{J \subseteq \{1,\ldots, m\}: \ m \in J}
	\Bigg( \prod_{i \in \{1,\ldots,\ell\}}
	\frac{G_\epsilon}{(j_i - j_{i-1})^{(1+\alpha-\frac{\epsilon}{2}) \zeta_\epsilon}} \Bigg) \,,
\end{equation}
where, recalling the definition \eqref{eq:Deps} of $D_\epsilon$, we have set
\begin{equation*}
	G_\epsilon \,:=\,
	(2 D_\epsilon )^{\zeta_\epsilon}
	\, e^{(1+\frac{\epsilon}{20})\frac{a_{\epsilon}^2}{2}
	\frac{\zeta_\epsilon}{1-{\zeta_\epsilon}} t_{\epsilon} }
	\,=\, \big( 2 (1+\epsilon) C_1 \big)^{\zeta_\epsilon}
	\, e^{\frac{\zeta_\epsilon}{\mu} (a_\epsilon - c_\epsilon)
	\tfrac{\epsilon}{5} t_\epsilon}
	\, e^{\big\{ \frac{\zeta_\epsilon}{\mu} (c_\epsilon - a_\epsilon) +
	(1+\frac{\epsilon}{20})\frac{a_{\epsilon}^2}{2}
	\frac{\zeta_\epsilon}{1-{\zeta_\epsilon}} \big\} t_{\epsilon} } \,.
\end{equation*}
Let us now replace the value of $a_\epsilon = \frac{1-\zeta_\epsilon}{\mu}$ that we
fixed in \eqref{eq:deltabetaeps} (recall from \eqref{eq:aopt} that this value is optimal 
in minimizing the second exponential, if we neglect the term $\epsilon/20$), getting
\begin{equation*}
	G_\epsilon\,=\,\big( 2 (1+\epsilon) C_1 \big)^{\zeta_\epsilon}
	\, e^{\frac{\zeta_\epsilon}{\mu} (\frac{1-\zeta_\epsilon}{\mu} - c_\epsilon)
	\tfrac{\epsilon}{5} t_\epsilon}
	\, e^{\frac{\zeta_\epsilon}{\mu}
	( c_\epsilon-(1-\frac{\epsilon}{20})\frac{1-\zeta_\epsilon}{2\mu} ) t_{\epsilon} }\,.
\end{equation*}
We now substitute in the value of 
$c_\epsilon = (1-\epsilon) \frac{\alpha}{1+\alpha} \frac{1}{2 \mu}$ set in \eqref{eq:fceps},
and substitute inside the parentheses in the exponential the value of 
$\zeta_\epsilon = \frac{1}{1+\alpha} + \frac{\epsilon}{2}\frac{\alpha}{1+\alpha}$ 
set in \eqref{eq:zetaeps}, which gives
\begin{equation*}
\begin{split}
	G_\epsilon &\,<\,\big( 2 (1+\epsilon) C_1 \big)^{\zeta_\epsilon}
	\, e^{\frac{\zeta_\epsilon}{\mu}
	(\frac{\alpha}{1+\alpha} \frac{1}{2 \mu})
	\tfrac{\epsilon}{5} t_\epsilon}
	\, e^{-\frac{\zeta_\epsilon}{\mu}
	\frac{9\epsilon}{20} \frac{\alpha}{1+\alpha} \frac{1}{2 \mu} t_{\epsilon} }
	\,=\, \big( 2 (1+\epsilon) C_1 \big)^{\zeta_\epsilon}
	\, e^{-\frac{\zeta_\epsilon}{\mu}
	\frac{\alpha}{1+\alpha} \frac{1}{2 \mu}
	\frac{\epsilon}{4} t_{\epsilon} }\,.
\end{split}
\end{equation*}

We are ready for the final step: by the definition \eqref{eq:zetaeps}
of $\zeta_\epsilon$, and the fact that $\epsilon$ has been fixed small enough,
we have $(1+\alpha-\frac{\epsilon}{2})\zeta_\epsilon > 1$.
Since the upper bound for $G_\epsilon$ vanishes as $t_\epsilon \to +\infty$, we can fix 
$t_\epsilon \in (0,\infty)$ large enough, depending only on $\epsilon$, such that
\begin{equation*}
	\sum_{n = 1}^\infty 
	\frac{G_\epsilon}{n^{(1+\alpha-\frac{\epsilon}{2})\zeta_\epsilon}}
	< 1\,.
\end{equation*}
The right hand side of \eqref{eq:pin} is then smaller than one,
because it can be recognized as the probability of visiting
$m$ for a renewal process, with return distribution given by
$\tilde K(n) := G_\epsilon/n^{(1+\alpha-\frac{\epsilon}{2})\zeta_\epsilon}$
and $\tilde K(\infty)=1-\sum_{n\in\N} \tilde K(n)>0$.
In conclusion, we have shown that for any $\epsilon>0$ small enough, we can find 
$\beta_0(\epsilon) := \min\{\beta_1(\epsilon), \beta_2(\epsilon), \beta_3(\epsilon),
\beta_4(\epsilon)\} \in (0,\infty)$ and $t_\epsilon \in (0,\infty)$, such that 
for all $\beta\in (0, \beta_0(\epsilon))$
\begin{equation*}
	\bbE \big[ \big( Z_{mt_\epsilon\beta^{-2},\beta, c_\epsilon \beta^2}^{\rc, \omega} 
	\big)^{\zeta_\epsilon} \big]
	\,\le\, 1 \qquad
	\forall \, m \in \N \,,
\end{equation*}
where $c_\epsilon$ was defined in \eqref{eq:fceps}. This establishes \eqref{eq:ttopprove} 
and concludes the proof.

\section{On the Pinning Model: Upper Bound}
\label{sec3}

\subsection{A lower bound on the free energy}
The strategy of the proof has been outlined in Section~\ref{S:outline}. First we prove 
the lower bound on the free energy of the pinning model stated in \eqref{eq:flowerbds}, 
which we restate here as a lemma.

\begin{lemma}
For every $c\in\R$
\begin{equation}\label{eq:gooa}
\liminf_{\beta\downarrow 0} \frac{\f(\beta, c\beta^2)}{\beta^2} 
\geq \frac{1}{\mu}\bigg[c-\frac{1}{2\mu} \bigg].
\end{equation}
\end{lemma}

\begin{proof}
A naive lower bound on the free energy is to apply Jensen's inequality,
interchanging the $\log$ in \eqref{eq:free} with the expectation $\bE$ over the renewal process
that appear in the partition function (recall \eqref{eq:Zpin}). 
However, this only leads to a trivial bound, as the expression
in the exponential in \eqref{eq:Zpin} 
is a linear function of the disorder $\omega$. 
To get a better bound, before applying Jensen
we perform a partial integration over a subset of the renewal points,
obtaining a ``coarse-grained Hamiltonian'' that is no longer linear in $\omega$. This 
has certain analogies with Theorem~5.2 in~\cite{cf:G1}.
The details are as follows.

For $q \in \N$ and $\ell \ge q$, we define $H^{(q)}_{\ell,\omega}$ to be the free energy of 
the constrained model of size $\ell$ conditioned to have exactly $q$ returns:
\begin{equation*}
	H^{(q)}_{\ell,\omega} \,:=\, \log \E\big[ e^{\sum_{n=1}^\ell
	(\beta \omega_n -\Lambda(\beta) + h) \ind_{\{n \in \tau\}}} \big| \tau_q = \ell \big] \,.
\end{equation*}
Let $\tau^{(q)} := \{\tau^{(q)}_n\}_{n\in\N_0}$ with $\tau^{(q)}_n := \tau_{nq}$, 
which is a renewal process that keeps one in every $q$ renewal points in $\tau$. We also set
\begin{equation*}
	\cN^{(q)}_N \,:=\, \max\{n \in \N_0: \ \tau^{(q)}_n \le N\}
	\,=\, |\tau^{(q)} \cap [1,N]| \,.
\end{equation*}
By requiring $N\in \tau^{(q)}$ and taking conditional expectation w.r.t. $\tau^{(q)}$, we obtain
\begin{equation*}
	Z_{N,\omega} \ge \E\bigg[ e^{\sum_{n=1}^N
	(\beta \omega_n -\Lambda(\beta) + h) \ind_{\{n \in \tau\}}} \ind_{\{N\in\tau^{(q)}\}} \bigg]  
	= \E\bigg[ \exp\bigg(\sum_{j=1}^{\cN^{(q)}_N}
	H^{(q)}_{\tau^{(q)}_j - \tau^{(q)}_{j-1}, \theta^{\tau^{(q)}_{j-1}} \omega}\bigg)
	\ind_{\{N\in\tau^{(q)}\}} \bigg] \,,
\end{equation*}
where $\theta^i\omega= \{\omega_{n+i}\}_{n\in\N}$ defines a shift of the disorder $\omega$.

Since $\e[\tau_1^{(q)}] = q \mu$ and $\tau_{\cN^{(q)}_N}^{(q)} \le N \le
\tau_{\cN^{(q)}_N +1}^{(q)}$, by the strong law of large numbers
\begin{equation} \label{eq:asdd}
	\lim_{N\to\infty} \frac{1}{N}
	\sum_{j=1}^{\cN^{(q)}_N} f(\tau^{(q)}_j - \tau^{(q)}_{j-1}) 
	\,=\, \frac{1}{q\mu} \, \e\big[f(\tau^{(q)}_1) \big] \,,
	\qquad \text{$\p$-a.s. and in $L^1(\dd\p)$}\,,
\end{equation}
for every function $f: \N \to \R$ such that $f(\tau^{(q)}_1) \in L^1(\dd \p)$.
Since $\P(N\in\tau^{(q)}) \to \frac{1}{q\mu} > 0$ as $N\to\infty$, by the renewal theorem,
it is not difficult to deduce that
\begin{equation} \label{eq:tecch}
	\lim_{N\to\infty}
	\e \Bigg[ \frac{1}{N} \sum_{j=1}^{\cN^{(q)}_N} f(\tau^{(q)}_j - \tau^{(q)}_{j-1}) \,\Bigg|\,
	N \in \tau^{(q)} \Bigg] \,=\, \frac{1}{q\mu} \, \e\big[f(\tau^{(q)}_1) \big] \,.
\end{equation}
We are going to apply this to $f(\ell) := \bbE [ H^{(q)}_{\ell, \omega}]$.
Recalling \eqref{eq:free}, by Jensen's inequality we get
\begin{equation} \label{eq:lbuse0}
\begin{split}
	\f(\beta,h) & \,\ge\, \limsup_{N\to\infty} \frac{1}{N}
	\bbE \E \Bigg[ \sum_{j=1}^{\cN^{(q)}_N}
	H^{(q)}_{\tau^{(q)}_j - \tau^{(q)}_{j-1}, \theta^{\tau^{(q)}_{j-1}} \omega}
	\Bigg| N\in\tau^{(q)} \Bigg]  \,=\, \frac{1}{q\mu} \bbE \bE
	\Big[ H^{(q)}_{\tau^{(q)}_1, \omega} \Big] \\
	& \,=\, \frac{1}{q\mu} \sum_{N=1}^\infty \P(\tau_q=N) \,
	\bbE \Big[ \log \E \Big[
	e^{\sum_{n=1}^N
	(\beta \omega_n -\Lambda(\beta)+ h) \ind_{\{n \in \tau\}}} \Big| \tau_q = N \Big] \Big] \\
    & \,=\, \frac{h-\Lambda(\beta)}{\mu} + \frac{1}{q\mu} \sum_{N=1}^\infty \P(\tau_q=N) \,
	\bbE \Big[ \log \E \Big[ e^{\sum_{n=1}^N
	\beta \omega_n  \ind_{\{n \in \tau\}}} \Big| \tau_q = N \Big] \Big]  \,.
\end{split}
\end{equation}
The rest of this section is devoted to studying this lower bound as $\beta, h \downarrow 0$.

\smallskip

Let us denote
\begin{equation}\label{eq:uu2}
		u_{N,q}(n) \,:=\, \P(n \in \tau | \tau_q = N) \,, \qquad u_{N,q}(n,m) 
		\,:=\, \P(n \in \tau, \, m \in \tau | \tau_q = N) \,. 
\end{equation}
By Jensen's inequality,
\begin{equation}\label{eq:logEJensen}
\log \E \Big[ e^{\sum_{n=1}^N\beta \omega_n  \ind_{\{n \in \tau\}}} \Big| \tau_q = N \Big] - \sum_{n=1}^N \beta\omega_n u_{N,q}(n)
 \geq 0.
\end{equation}
Therefore we can apply Fatou's Lemma in \eqref{eq:lbuse0} to obtain
\begin{eqnarray}
&& \liminf_{\beta\downarrow 0} \frac{\f(\beta, c\beta^2)}{\beta^2} \label{eq:fbetacbd}\\
&=&\frac{c-\frac{1}{2}}{\mu}+ \liminf_{\beta\downarrow 0} \frac{1}{q\mu} 
\sum_{N=1}^\infty \P(\tau_q=N) \,	
\bbE \Big[ \frac{\log \E \big[ e^{\sum_{n=1}^N \beta \omega_n  \ind_{\{n \in \tau\}}} 
\big| \tau_q = N \big]-\sum_{n=1}^N \beta\omega_n u_{N,q}(n)}{\beta^2} \Big]  \nonumber \\
&\geq& \frac{c-\frac{1}{2}}{\mu}+\frac{1}{q\mu} \sum_{N=1}^\infty \P(\tau_q=N) \,	
\bbE \Big[ \liminf_{\beta\downarrow 0}\frac{\log \E \big[ e^{\sum_{n=1}^N \beta \omega_n 
 \ind_{\{n \in \tau\}}} \big| \tau_q = N \big]-\sum_{n=1}^N \beta\omega_n u_{N,q}(n)}{\beta^2} \Big]. \nonumber
\end{eqnarray}
By Taylor expansion, for fixed disorder $\go$ and as $\gb\downarrow 0$, we have
\begin{equation*}
\begin{split}
	\E \Big[
	e^{\sum_{n=1}^N
	\beta \omega_n \ind_{\{n \in \tau\}}} \Big| \tau_q = N \Big]
	\,=\,  1 + \beta \sum_{n=1}^N
	\omega_n u_{N,q}(n)
	 + \frac{1}{2} \beta^2
	\sum_{m,n=1}^N \omega_m \omega_n u_{N,q}(m,n) + o(\beta^2) \,.
\end{split}
\end{equation*}
Since $\log(1+x) = x - \frac{1}{2} x^2 + o(x^2)$ as $x\downarrow 0$, we obtain
\begin{multline}
 \bbE\Big[\liminf_{\beta\downarrow 0}\frac{\log \E \big[ e^{\sum_{n=1}^N \beta \omega_n  
 \ind_{\{n \in \tau\}}} \big| \tau_q = N \big]-\sum_{n=1}^N \beta\omega_n u_{N,q}(n)}{\beta^2}\Big] \\
= \bbE\Big[\frac{1}{2}\sum_{m,n=1}^N \omega_m\omega_n\big(u_{N,q}(m,n)-
u_{N,q}(m)u_{N,q}(n)\big)\Big] \\
= \frac{1}{2} \sum_{n=1}^N (u_{N,q}(n) - u_{N,q}(n)^2) = \frac{q}{2} 
- \frac{1}{2}\sum_{n=1}^Nu_{N,q}(n)^2 \,, \label{eq:Elimbetabd}
\end{multline}
where the last equality holds, by \eqref{eq:uu2}, 
because $\sum_{n=1}^N \ind_{\{n \in \tau\}} = q$
on the event $\{\tau_q = N\}$.

Note that
\begin{equation}\label{eq:sumsq}
	\sum_{n=1}^Nu_{N,q}(n)^2 =  \E\big[ |\tau\cap\tilde\tau\cap(0,N]| \big|\tau_q=\tilde\tau_q=N\big],
\end{equation}
where $\tilde\tau$ is an independent copy of $\tau$. Intuitively, since each renewal process
$\tau, \tilde\tau$ has mean return time $\mu$, the expression
in \eqref{eq:sumsq} should be of the order $q/\mu$. In order to prove it,
we fix $\eta>0$. Decomposing the right hand side in \eqref{eq:sumsq} according 
to whether $|\tau\cap\tilde\tau\cap (0,N] | \leq (1+\eta) N/\mu^2$ or not, and
noting that $|\tau\cap(0,N]| \, \ind_{\{\tau_q=N\}} = q$,
we obtain
\begin{eqnarray*}
\sum_{n=1}^Nu_{N,q}(n)^2 &\leq& \frac{(1+\eta)N}{\mu^2}+ q 
\frac{\P\big(|\tau\cap\tilde\tau\cap (0,N] 
| > (1+\eta) N/\mu^2, \tau_q=\tilde\tau_q=N \big) }{\P(\tau_q=\tilde\tau_q=N)} \\
&\leq& \frac{(1+\eta)N}{\mu^2}+ q 
\frac{\sqrt{\P\big(|\tau\cap\tilde\tau\cap (0,N] | 
> (1+\eta) N/\mu^2\big)} \cdot \P(\tau_q=N) }{\P(\tau_q=N)^2}\,,
\end{eqnarray*}
where we used Cauchy-Schwarz inequality for the second inequality.
We note that $\tau\cap\tilde \tau$ is a renewal process with finite mean $\mu^2$. 
Therefore, by a standard Cramer large deviation estimate \cite[Theorem 2.2.3]{cf:DZ},
there exist $C_\eta \in (0,\infty)$ such that for all $q\in\N$ large enough
$$
	\max_{N\in ((1-\eta)q\mu, (1+\eta)q\mu)} \P\big(|\tau\cap\tilde\tau\cap (0,N] 
	| > (1+\eta) N/\mu^2\big) \leq e^{-C_\eta \, q},
$$
and hence, uniformly in $N\in ((1-\eta)q\mu, (1+\eta)q\mu)$, we have
\begin{equation}\label{eq:alfb}
\sum_{n=1}^Nu_{N,q}(n)^2 \leq q\frac{(1+\eta)^2}{\mu}
+ q \frac{e^{-\frac{1}{2} C_\eta \, q}}{\P(\tau_q=N)}.
\end{equation}
We finally plug the bound \eqref{eq:alfb} into \eqref{eq:Elimbetabd}, 
and then into \eqref{eq:fbetacbd} (note that the denominator $\P(\tau_q=N)$ in \eqref{eq:alfb}
gets simplified). Restricting the summation to $N\in ((1-\eta)q\mu, (1+\eta)q\mu)$,
thanks to \eqref{eq:logEJensen}, we obtain, for $q$ large enough,
$$
\liminf_{\beta\downarrow 0} \frac{\f(\beta, c\beta^2)}{\beta^2} \geq  \frac{c-\frac{1}{2}}{\mu} 
+ \frac{1}{2\mu} \Big(1- \frac{(1+\eta)^2}{\mu}\Big) p_\eta(q) 
-\frac{1}{2\mu}(2\eta q\mu)\, e^{-\frac{1}{2} C_\eta \, q},
$$
where $p_\eta(q):=\P\big(\tau_q \in ((1-\eta)q\mu, (1+\eta)q\mu) \, \big)\to 1$ as $q\to\infty$ 
by the law of large numbers.
First letting $q\uparrow \infty$ and then letting $\eta\downarrow 0$ gives 
the desired bound \eqref{eq:gooa}.
\end{proof}

\subsection{Completing the proof}
\label{sec:completing}

Recall the smoothing inequality \eqref{eq:smooth}
\begin{equation}
	0\leq \f^{{\rm pin}}(\beta, h) \,\le\, \frac{1+\alpha}{2} A_{\gb,\frac{h-h_c(\gb)}{\beta}} 
	\frac{(h-h_c(\beta))^2}{\beta^2} \,.
\end{equation}
Observe that $0 \le h_c(\beta) \le \Lambda(\beta) = \frac{1}{2} \beta^2 + o(\beta^2)$
for every $\beta \ge 0$, cf. \cite[Proposition 5.1]{cf:G1}.
Setting $h = c \beta^2$, we then have
$\frac{h-h_c(\gb)}{\beta} \to 0$ and hence, since
$\lim_{(\beta,\delta) \to (0,0)} A_{\beta,\delta} = 1$,
\begin{equation}\label{eq:comb01}
	\liminf_{\beta\downarrow 0} \frac{\f(\beta, c\beta^2)}{\beta^2} \leq \frac{1+\alpha}{2} 
\bigg[ c - \bigg(\limsup_{\beta \downarrow 0}
	\frac{h_c(\beta)}{\beta^2} \bigg) \bigg]^2 \,.
\end{equation}
Combining \eqref{eq:comb01} with \eqref{eq:gooa} gives
\begin{equation*}
	\frac{1+\alpha}{2} \bigg[ c - \bigg(\limsup_{\beta \downarrow 0}
	\frac{h_c(\beta)}{\beta^2} \bigg) \bigg]^2 \,\ge\,
	\frac{1}{\mu}\bigg( c - \frac{1}{2\mu} \bigg)
	 \qquad \forall\, c \in \R \,.
\end{equation*}
We can rewrite this inequality as
\begin{equation*}
	A c^2 + B c + C \ge 0  \qquad \forall\, c \in \R \,,
\end{equation*}
with
\begin{gather*}
	A = \frac{1+\alpha}{2} \,, \quad \
	B = -(1+\alpha) \limsup_{\beta \downarrow 0}
	\frac{h_c(\beta)}{\beta^2} -\frac{1}{\mu} \,, \quad \
	C = \frac{1+\alpha}{2} \bigg(\limsup_{\beta \downarrow 0}
	\frac{h_c(\beta)}{\beta^2} \bigg)^2
	+ \frac{1}{2\mu^2} \,.
\end{gather*}
Then we must have $B^2 - 4AC \le 0$ and this readily leads to
\begin{equation*}
	\limsup_{\beta \downarrow 0}
	\frac{h_c(\beta)}{\beta^2} \,\le\,
	  \frac{1}{2\mu} \frac{\alpha}{1+\alpha} \,,
\end{equation*}
which is precisely the upper bound in \eqref{eq:CSZ+}.

\section{On the Copolymer Model: Lower Bound}\label{sec4}

We now consider the copolymer model, with constrained partition function
$$
	Z_{N,\gl,h}^{{\rm cop}, \rc, \omega} \,=\,
	\E\Big[ e^{-2\gl\sum_{n=1}^N(\omega_n +h_a(\lambda) - h) 
	\ind_{\{\varepsilon_n = -1\}} } \ind_{\{ N\in\tau\}}\Big] \,.
$$
where we recall that $h_a(\lambda) =
h_{a}^{{\rm cop}}(\lambda) =(2\gl)^{-1}\Lambda(-2\gl)$, cf. \eqref{eq:hanncop}.
The method and steps are the same as for the pinning model, discussed in Section~\ref{sec2},
with only minor differences. In fact, replacing 
$\omega$ by $-\omega$, $2\gl$ by $\beta$ and $2\lambda h$ by $h$ 
casts the copolymer partition function in exactly 
the same form as the random pinning model, the only difference being $\ind_{\{n\in\tau\}}$ 
in the pinning partition function replaced by $\ind_{\{\varepsilon_n = -1\}}$, cf.
\eqref{eq:Zpinc}.

Relations \eqref{eq:tildeZ}, \eqref{eq:lkj}
and \eqref{eq:weakco} still hold with $|\tau\cap [1, k]|$ replaced by 
$\sum_{n=1}^k 1_{\{\epsilon_n = -1\}}$ and $\mu$ replaced by $2$. Following the 
same procedure, it suffices to show
\begin{equation}
	\liminf_{N\to\infty} \bbE\big[ \big( Z_{N,\gl,h}^{\rc, \omega} \big)^\zeta \big] \,<\, \infty
\end{equation}
for
\begin{eqnarray}\label{cophz}
	h=h_\epsilon:=c_\epsilon\lambda :=(1-\epsilon)\frac{\ga}{2(\ga+1)}\gl 
	\qquad \text{and} \qquad
	\zeta=\zeta_\epsilon \,:=\, \frac{1}{1+\alpha} \,+\, \frac{\epsilon}{2}\frac{\alpha}{1+\alpha}
	 \,.
\end{eqnarray}
The only major difference in the calculation is in \eqref{ZJpin}. In the copolymer case, this is replaced by
\begin{equation}\label{ZJcop}
\begin{split}
	\hat Z_J  \,:=\, \sumtwo{d_1, f_1 \in B_{j_1}}{d_1 \le f_1}
	\ldots \sumtwo{d_{\ell-1}, f_{\ell-1} \in B_{j_{\ell-1}}}{d_{\ell - 1} \le f_{\ell - 1}}
	\sum_{d_\ell \in B_{j_\ell} = B_m}
	\Bigg( \prod_{i=1}^\ell K(d_i - f_{i-1}) z_{f_{i-1},d_i}^{{\rm cop}} Z_{d_i, f_i} \Bigg) \,,
\end{split}
\end{equation}
where $Z_{d,f}$ is defined in analogy to \eqref{zdf} and $z_{f_{i-1},d_i}^{{\rm cop}} := z_{(f_{i-1},d_i]}^{{\rm cop}}$, with
\begin{equation*}
z_I^{{\rm cop}} :=\frac{1+e^{-2\lambda \sum_{n\in I\cap\N}(\omega_n 
+h_a(\lambda) - h)}}{2} \qquad \mbox{for any } I\subset (0,\infty).
\end{equation*}
Defining $\overline{f}_{i-1}:=j_{i-1}k$ and $\overline{d}_{i}:=(j_i-1) k$, with 
$k= t_\epsilon \beta^{-2}=t_\epsilon \lambda^{-2}/4$, we use that (see \cite[(3.16)]{cf:T1})
\begin{eqnarray}\label{prodz}
z_{f_{i-1},d_i}^{{\rm cop}}\leq 2 \,z_{(f_{i-1},\overline{f}_{{i-1}}]\cup (\overline{d}_{i},d_i] }^{{\rm cop}}  \,\,z_{\overline{f}_{{i-1}}, \overline{d}_{i} }^{{\rm cop}}.
\end{eqnarray}
Following \eqref{eq:fracmom} we have that
\begin{eqnarray} \label{eq:fracmomcop}
	\bbE \big[ \big( Z_{N,\gl, c_\epsilon \gl}^{ \rc, \omega} \big)^{\zeta_\epsilon} \big]
	\,\le\, \sum_{J \subseteq \{1,\ldots, m\}:
	\ m \in J} \bbE \big[ \big( \hat Z_J \big)^{\zeta_\epsilon} \big] \,,
\end{eqnarray}
for $\gz_\epsilon$ chosen in \eqref{cophz} (the same as in \eqref{eq:zetaeps}).
Substituting \eqref{prodz} into \eqref{ZJcop}, we have
\begin{eqnarray*}
                \bbE \big[ \big( \hat Z_J \big)^{\zeta_\epsilon} \big] \leq
                2^\ell \prod_{i=1}^\ell \bbE \left[\left(z_{\overline{f}_{i-1},\overline{d}_i}^{{\rm cop}} \right)^{\zeta_\epsilon} \right]
                 \bbE \big[ \big( \breve Z_J \big)^{\zeta_\epsilon} \big],
\end{eqnarray*}
where
$$
	\breve Z_J  \,:=\, \sumtwo{d_1, f_1 \in B_{j_1}}{d_1 \le f_1}
	\ldots \sumtwo{d_{\ell-1}, f_{\ell-1} \in B_{j_{\ell-1}}}{d_{\ell - 1} \le f_{\ell - 1}}
	\sum_{d_\ell \in B_{j_\ell} = B_m}
	\Bigg( \prod_{i=1}^\ell K(d_i - f_{i-1}) z_{(f_{i-1},\overline{f}_{{i-1}}]\cup (\overline{d}_{i},d_i] }^{{\rm cop}} Z_{d_i, f_i} \Bigg) \,. 
$$
To proceed further, one needs to note that
$(a+b)^{\zeta_\epsilon} \le a^{\zeta_\epsilon} + b^{\zeta_\epsilon}$, for all $a,b\ge 0$, hence
\begin{equation*}
\begin{split}
	\bbE \left[\left(z_{\overline{f}_{i-1},\overline{d}_i}^{{\rm cop}} \right)^{\zeta_\epsilon} 
	\right] & \leq \frac{1}{2^{\zeta_\epsilon}} \Big( 1 + 
	\bbE \Big[e^{-2 \lambda \zeta_{\epsilon} \sum_{n =\overline{f}_{i-1}+1}^{\overline{d}_{i}} 
	(\omega_n + h_a(\lambda) - h)} \Big] \Big)\\
	& = \frac{1}{2^{\zeta_\epsilon}} \Big( 1 + 
	\bbE \Big[e^{(\overline{d}_{i} - \overline{f}_{i-1})
	[\Lambda(-2 \lambda \zeta_{\epsilon}) -
	\zeta_{\epsilon}\Lambda(-2 \lambda)  + 2 \lambda \zeta_\epsilon h]} \Big] \Big)
	\leq \frac{1}{2^{\zeta_\epsilon}} ( 1 + 1) = 2^{1-\zeta_\epsilon},
\end{split}
\end{equation*}
where the last inequality
holds for $\lambda$ and $\epsilon$ small enough. Indeed, (recall \eqref{eq:asLambda} and 
\eqref{cophz})
\begin{equation}
\Lambda(-2\lambda\zeta_{\epsilon}) - \zeta_{\epsilon} \Lambda(-2\lambda) 
+ 2\lambda \zeta_{\epsilon} h \stackrel{\lambda\downarrow 0}{\sim} 
2\lambda^2 \zeta_{\epsilon} (\zeta_{\epsilon} -1 + c_{\epsilon})
\end{equation}
and 
\begin{equation}
\lim_{\epsilon\downarrow 0}(\zeta_{\epsilon} -1 + c_{\epsilon}) 
= \frac{-\alpha}{2(1+\alpha)} < 0.
\end{equation}
Finally, let $\tilde\bbP_{J}$ be the law of the disorder obtained from $\bbP$, where 
independently for each $n \in \bigcup_{i \in J} B_i$, the law of $\omega_n$ is tilted with 
density $e^{\delta \omega_n-\Lambda(\delta)}$, with
\begin{equation*}
	\gd:=\, a_\epsilon \beta 
	:= (1-\zeta_\epsilon) \gl \,,
\end{equation*}
cf. \eqref{eq:deltabetaeps}. In complete analogy with \eqref{zn}, we have
\begin{eqnarray*}
\tilde\bbE_{J}\left[ z_{(f_{i-1},\overline{f}_{{i-1}}]
\cup (\overline{d}_{i},d_i] }^{{\rm cop}} \right] \leq 1.
\end{eqnarray*}
The rest of the proof then proceeds exactly as in the analysis of the pinning model.

\section{On the Copolymer Model: Upper Bound}
\label{sec5}

The proof goes along the very same lines as for the pinning model,
cf. Section~\ref{sec3}. In fact, the analogue of the lower bound
\eqref{eq:gooa} on the free energy is much simpler for the copolymer.

\begin{lemma}\label{th:Jencop}
For every $c\in\R$
\begin{equation}\label{eq:gooa2}
\liminf_{\gl\downarrow 0} \frac{\f(\lambda, c\lambda)}{\lambda^2} \geq c-\frac{1}{2}.
\end{equation}
\end{lemma}
\begin{proof}
A direct application of Jensen's inequality is sufficient. Let
$$
\cN_N:=\max\{n\in \bbN_0\colon \tau_n\leq N\} = |\tau \cap [1, N]|.
$$
Recalling \eqref{eq:free} and \eqref{eq:Zcopc}, in analogy with \eqref{eq:tecch} we obtain
\begin{align*}
	\f(\lambda,h) &=
	\lim_{N\to \infty} \frac{1}{N}\,\bbE\log \E\Big[ \prod_{j=1}^{\cN_N}
	\frac{1+e^{-2\lambda\sum_{n=\tau_{j-1}+1}^{\tau_j} (\go_n+h_{a}(\lambda) -h) } }{2} \Big|
	\, N\in\tau \Big]\\
	&\geq \frac{1}{\mu} \bbE\E \left[\log\frac{1+e^{-2\lambda\sum_{n=1}^{\tau_1} 
	(\go_n+h_{a}(\lambda) -h) } }{2}  \right]\\
	&= \frac{1}{\mu} \sum_{N=1}^\infty K(N)\, \bbE \left[\log\frac{1+e^{-2\lambda\sum_{n=1}^{N} 
	(\go_n+h_{a}(\lambda)-h) } }{2} + \lambda\sum_{n=1}^N(\go_n+h_{a}(\lambda)-h) \right]\\
    & \quad \,\,-\gl(h_{a}(\lambda)- h),
\end{align*}
where the term $\lambda\sum_{n=1}^N(\go_n+h_{a}(\lambda)-h)$ 
is inserted to ensure that the expression inside the expectation is nonnegative,
by Jensen's inequality. We can then apply Fatou's Lemma, analogously to \eqref{eq:fbetacbd}:
recalling \eqref{eq:hanncop}, a simple Taylor expansion yields
\begin{eqnarray*}
\liminf_{\gl\downarrow 0}\frac{\f(\gl,c\gl)}{\lambda^2}\geq  c-\frac{1}{2},
\end{eqnarray*}
completing the proof.
\end{proof}

Coupling the lower bound \eqref{eq:gooa2}
with the smoothing inequality \eqref{eq:smoothcop} for the copolymer model,
exactly as we did for the pinning model in Section~\ref{sec:completing},
we obtain 
\begin{equation*}
	\limsup_{\lambda \downarrow 0}
	\frac{h_c(\lambda)}{\lambda} \,\le\,
	\frac{\alpha}{2(1+\alpha)} \,,
\end{equation*}
which completes the proof of Theorem~\ref{thmcop}.

\bigskip
\noindent
{\bf Acknowledgments.} RS is supported by NUS grant R-146-000-148-112. NZ is 
supported by a Marie Curie International Reintegration Grant within the 7th European Community Framework Programme, IRG-246809. FC and JP acknoweldge the support of ERC Advanced Grant 267356 VARIS and ANR MEMEMO2 10-BLAN-0125-03.
RS, FC and NZ thank the Department of Mathematics and
Applications at the University of Milano-Bicocca,  the Institute for Mathematical Sciences and the 
Department of Mathematics at the National University 
of Singapore, and Academia Sinica in Taipei for hospitality, 
where parts of this work were carried out. 
 QB is grateful to ENS Lyon for its support while this work was initiated.


\normalsize

\bigskip

\begin{thebibliography}{AA}



\bibitem{cf:A2}
Alexander, K.S.
\textit{The effect of disorder on polymer depinning transitions},
Commun.\
Math.\ Phys.\ , 279 (2008), 117--146.



\bibitem{cf:A}
Alexander, K.S.
\textit{Excursions and local limit theorems for Bessel-like random walks},
Electr.\ J.\ Prob.\ 16 (2011), 1--44.

\bibitem{cf:AZ}
Alexander, K.S., Zygouras, N.
\textit{Quenched and annealed critical points in polymer pinning models},
Comm.\ Math.\ Phys.\ 291 (2010), 659--689.

\bibitem{cf:AZ2}
Alexander, K.S., Zygouras, N.
\textit{Equality of Critical Points for Polymer Depinning Transitions with Loop Exponent One},
Ann. Appl. Prob \textbf{20} (2010), 356-366.

\bibitem{cf:BG}
T. Bodineau and  G. Giacomin, 
\textit{On the localization transition  of random copolymers near selective interfaces},
J. Statist. Phys. {\bf 117} (2004), 801--818.

\bibitem{cf:BGLT}
Bodineau, T., Giacomin, G., Lacoin, H., Toninelli, F.L.
\textit{Copolymers at selective interfaces: new bounds on the phase diagram},
J.\ Stat.\ Phys.\ 132 (2008), 603-–626.

\bibitem{cf:BdH}
Bolthausen, E., den Hollander, F.
\textit{Localization transition for a polymer near an interface},
Ann.\ Probab.\ 25 (1997), 1334--1366.

\bibitem{cf:BdHO}
Bolthausen, E., den Hollander, F., Opoku, A.A.
\textit{A copolymer near a selective interface: variational characterization of the free energy},
Ann. Probab (to appear), arXiv:1110.1315 (2011).

\bibitem{cf:Ca}
Caravenna, F., den Hollander, F.,
\textit{A general smoothing inequality for disordered polymers},
arXiv:1306.3449 (2013).

\bibitem{cf:CG}
Caravenna F., Giacomin G.
\textit{The weak coupling limit of disordered copolymer models},
Ann.\ Probab.\ 38 (2010), 2322--2378.

\bibitem{cf:CGG}
Caravenna, F., Giacomin, G., Gubinelli, M.
\textit{A numerical approach to copolymers at selective interfaces},
J. Stat. Phys. 122 (2006), 799-832.

\bibitem{cf:CGT}
Caravenna, F., Giacomin, G., Toninelli, F.L.
\textit{Copolymers at selective interfaces: settled issues and open problems},
Probability in Complex Physical Systems--In honour of Erwin Bolthausen and J\"urgen G\"artner, 289--311. Springer Proceedings in Mathematics 11 (2012).

\bibitem{cf:CSZ}
Caravenna, F., Sun, R., Zygouras, N.
\textit{Continuum limits of the random pinning model under weak coupling},
in preparation.

\bibitem{cf:CdH}
Cheliotis, D., den Hollander, F.
\textit{Variational characterization of the critical curve for pinning of random polymers},
Ann. Probab. 41 (2013), 1767-1805.

\bibitem{cf:DZ} A. Dembo, O. Zeitouni,
\emph{Large Deviations Techniques and Applications} (2nd. ed.),
Springer (1998).

\bibitem{cf:DGLT}
Derrida, B., Giacomin, G., Lacoin, H., Toninelli, F.L.
\textit{Fractional moment bounds and disorder relevance for pinning models},
Comm.\ Math.\ Phys.\ 287 (2009), 867–-887.

\bibitem{cf:DHV}
Derrida, B., Hakim, V. and Vannimenius, J.
\textit{Effect of disorder on two dimensional wetting},
J.\ Stat.\ Phys.\ 66 (1992), 1189–-1213.

\bibitem{cf:FLNO}
Forgacs, G., Luck, J. M., Nieuwenhuizen, Th. M. and Orland, H.
\textit{Wetting of a disordered substrate: Exact critical behavior in two dimensions},
Phys.\ Rev.\ Lett.\ 57 (1986), 2184–-2187.

\bibitem{cf:GHLO}
Garel, T., Huse, D. A., Leibler, S. and Orland, H.
\textit{Localization transition of random chains at interfaces},
Europhys.\ Lett.\ 8 (1989), 9–-13.

\bibitem{cf:G1}
Giacomin, G.
\textit{Random polymer models},
Imperial College Press (2007).

\bibitem{cf:G2}
Giacomin, G.
\textit{Disorder and critical phenomena through basic probability models},
Lecture notes from the 40th Probability Summer School held in Saint-Flour, 2010.
Springer (2011).

\bibitem{cf:GLT1}
Giacomin, G., Lacoin, H., Toninelli, F.L.
\textit{Disorder relevance at marginality and critical point shift},
Ann.\ Inst.\ H.\ Poincar\'e Probab.\ Stat.\ 47 (2011), 148--175.

\bibitem{cf:GLT2}
Giacomin, G., Lacoin, H., Toninelli, F.L.
\textit{Marginal relevance of disorder for pinning models},
Comm.\ Pure Appl.\ Math.\ 63 (2011), 233--2650.

\bibitem{cf:GT}
Giacomin, G., Toninelli, F.L.
\textit{Smoothing effect of quenched disorder on polymer depinning transitions},
Commun.\ Math.\ Phys.\ 266 (2006), 1--16.

\bibitem{cf:dH}
den Hollander, F.
\textit{Random polymers}.
Lectures from the 37th Probability Summer School held in Saint-Flour, 2007.
Springer-Verlag, Berlin (2009).

\bibitem{cf:L}
Lacoin, H.
\textit{The martingale approach to disorder irrelevance for pinning models}, 
Electron. Comm. Probab. \textbf{15} (2010), 418-427.

\bibitem{cf:M}
Monthus, C.
\textit{On the localization of random heteropolymers at the interface between two selective solvents},
Eur.\ Phys.\ J.\ B 13 (2000), 111–-130.

\bibitem{cf:NV}
Nelson, D.R., Vinokur, V.M.,
\textit{Boson localization and correlated pinning of superconducting vortex arrays},
Phys.\ Rev.\ B 48 (1993), 13060-–13097.

\bibitem{cf:PS}
Poland, D., Scheraga, H.
\textit{Theory of helix-coil transitions in biopolymers: statistical mechanical theory of order-disorder transitions in biological macromolecules},
Academic Press (1970).

\bibitem{cf:T2}
Toninelli, F.L.
\textit{A replica-coupling approach to disordered pinning models}, Commun.\
Math.\ Phys.\ , 280 (2008), 389--401

\bibitem{cf:T1}
Toninelli, F.L.
\textit{Coarse graining, fractional moments and the critical slope of random copolymers},
Electron.\ J. Probab.\ 14 (2009), 531--547.



\end{thebibliography}
\end{document}